\documentclass[11pt,twoside]{amsart}
\usepackage{latexsym,amssymb,amsmath}
\usepackage[all]{xy}
\usepackage{appendix}

\numberwithin{equation}{section}
\newtheorem{theorem}{Theorem}[section]
\newtheorem{lemma}[theorem]{Lemma}
\newtheorem{proposition}[theorem]{Proposition}
\newtheorem{corollary}[theorem]{Corollary}

\theoremstyle{definition}
\newtheorem{definition}[theorem]{Definition}
\newtheorem{procedure}[theorem]{Procedure}
\newtheorem{remark}[theorem]{Remark}
\newtheorem{example}[theorem]{Example}

\begin{document}

\title[Signature invariants]{Signature invariants of monomial
ideals}

\thanks{The first author was supported by a scholarship from SECIHTI,
M\'exico. The second and third authors were supported by SNII, M\'exico.}

\author[J. Ibarguen]{Jovanny Ibarguen}
\address{
Departamento de Matem\'aticas\\
Cinvestav, Av. IPN 2508, 07360, CDMX, M\'exico.
}
\email{jibarguen@math.cinvestav.mx}

\author[C. E. Valencia]{Carlos E. Valencia}
\address{
Departamento de Matem\'aticas\\
Cinvestav, Av. IPN 2508, 07360, CDMX, M\'exico.
}
\email{cvalencia@math.cinvestav.edu.mx}
\author[R. H. Villarreal]{Rafael H. Villarreal}
\address{
Departamento de Matem\'aticas\\
Cinvestav, Av. IPN 2508, 07360, CDMX, M\'exico.
}
\email{rvillarreal@cinvestav.mx}


\begin{abstract} 
Let $I$ be a monomial ideal of a polynomial ring $R=K[x_1,\ldots,x_n]$
over a field $K$ and let ${\rm sgn}(I)$ be its signature ideal. If $I$
is not a principal ideal, we show that the depth of $R/I$ is the depth
of $R/{\rm sgn}(I)$, and the regularity of
$R/{\rm sgn}(I)$ is at most the regularity of $R/I$. For ideals of height at least $2$, we 
show that the associated primes of $I$ and ${\rm sgn}(I)$ are the same, and we show 
that $I$ is Cohen--Macaulay (resp. Gorenstein) if and only if ${\rm sgn}(I)$ is
Cohen--Macaulay (resp. Gorenstein), and furthermore we show that the v-number of
${\rm sgn}(I)$ is at most the v-number of $I$ and compare the irreducible
decompositions of $I$ and ${\rm sgn}(I)$. We give an algorithm 
to compute the signature of a
monomial ideal using \textit{Macaulay}$2$, and an algorithm 
to examine given families of monomial ideals by computing 
their signature ideals and determining which of these are
Cohen--Macaulay or Gorenstein.  
\end{abstract}

\maketitle

\section{Introduction}\label{intro-section}

Let $R=K[x_1,\ldots,x_n] $ be a polynomial ring over a field $K$ and
let $\mathbb{N}=\{0,1,2,\dots\}$. 
To make notation simpler, we use the following
multi-index notation to denote the monomials of $R$: 
$$
x^{a}:=x_1^{a_{1}}\cdots x_n^{a_{n}}\ \mbox{ for }
a=(a_{1},\ldots,a_n)\in \mathbb{N}^n,
$$
and define the \textit{support} of $x^a$ as ${\rm
supp}(x^a):=\{x_i\mid a_i>0\}$. 
The set $\mathbb{N}^n$ is a poset $(\mathbb{N}^n,\,\leq)$ under the usual
componentwise order, and equivalently the set of monomials of $R$, denoted by $\mathbb{M}_n$, is 
a poset ($\mathbb{M}_n$,\,$\preceq$) under divisibility. 
There is an isomorphism given by  
$$
\mathbb{N}^n\longrightarrow\mathbb{M}_n,\quad 
a\mapsto x^a,
$$
between the additive semigroup $(\mathbb{N}^n,\,+)$ and the
multiplicative semigroup $(\mathbb{M}_n,\,\cdot\,)$. The set of
monomial ideals of $R$ is denoted by $\mathcal{M}_n(R)$. 
Let $I$ be a monomial ideal of $R$. 
Note that, by Dickson's lemma \cite[Lemma 3.3.3]{monalg-rev}, 
the ideal $I$ is always minimally generated by a unique finite set $G(I)$ of 
monomials:
$$G(I):=\{x^{v_1},\ldots,x^{v_q}\}.$$
\quad The \textit{incidence matrix} of $I$, 
denoted by $A$, is the 
$n\times q$ matrix with column vectors $v_1,\ldots,v_q$. An \textit{anti-chain} of
($\mathbb{M}_n$,\,$\preceq$) is a set of non-comparable
monomials which is necessarily finite by Dickson's lemma. There is a one-to-one
correspondence between the set of monomial ideals $\mathcal{M}_n(R)$ and the set 
of anti-chains of $\mathbb{M}_n$:
$$
I \longmapsto G(I),
$$ 
and for this reason a monomial ideal is called an
\textit{anti-chain} ideal. 

The signature of $I$ was introduced by Ibarguen et al.
\cite{signature} to study 
its irreducible decomposition and algebraic invariants. For
convenience, we first define the signature of its incidence matrix. 

The \textit{signature} of $A$, 
denoted ${\rm sgn}(A)$, is the $n\times q$ matrix constructed as
follows. For each row $c=\{c_1,\ldots,c_q\}$ of $A$, we list the
distinct entries of $c$ in ascending order 
\begin{equation}\label{nov27-25}
m=\{m_0,m_1,\ldots,m_r\},\ m_0<m_1<\cdots<m_r,\ \, 0\leq r\leq q-1.
\end{equation} 
\quad Then, for each $c_i$ consider the position $j_i$
of $c_i$ relative to $m$, that is, $c_i=m_{j_i}$, $0\leq j_i\leq r$.
Letting ${\rm sgn}(c):=\{j_1,\ldots,j_q\}$, the \textit{signature} ${\rm
sgn}(A)$ of $A$ is
the matrix with rows ${\rm sgn}(c)$ with $c$ running over all rows of
$A$. Note that each entry of ${\rm sgn}(A)$ is at most $q-1$.

The \textit{signature} of a
monomial ideal $I$ with incidence matrix $A$, denoted by ${\rm
sgn}(I)$, is the ideal of $R$ generated by the 
monomials corresponding to the columns of   
${\rm sgn}(A)$. Note that the 
degree of any minimal generator of ${\rm sgn}(I)$ is at most $n(q-1)$.
If $I$ is a monomial ideal and $I$ 
is not principal, then ${\rm sgn}(I)$ has height at least $2$
(Proposition~\ref{sgn-ht}). 

Some of the algebraic invariants and combinatorial properties of a monomial
ideal $I$ are preserved when taking the signature
\cite{2025-MFRSignature,signature,2025-RegularSequences}. 
If the height of the ideal $I$ is at least $2$, the signature reduces the
degrees of the monomials in $G(I)$ while  
leaving invariant the supports of the monomials in $G(I)$. 
This is different to algebraic shifting, where the idea
is to associate with a simplicial 
complex $\Delta$ a shifted simplicial complex that shares some
homological properties with $\Delta$, e.g., has the same Betti
numbers. 
The shifting technique is rich in connections and
applications. There is also the notion of combinatorial shifting 
(see \cite{kalai} and references therein). 

The columns of ${\rm sgn}(A)$ form an anti-chain
in $\mathbb{N}^n$ \cite[Lemma~3.7]{signature}
and the rows of
${\rm sgn}(A)$ are \textit{tight} in the sense that,  
for any $1\leq
i\leq n$, the set of distinct entries of 
the $i$-th row of ${\rm sgn}(A)$, listed in ascending order, is
$\{0,1,2,\ldots,\ell_i\}$ for some $\ell_i$. Equivalently, from
Eq.~\eqref{nov27-25}, a row $c=(c_1,\ldots,c_q)$ of $A$ is not tight 
if there is $0\leq k<r$ such that $m_{k+1}> m_{k}+1$. 
If $I$ is a squarefree monomial ideal of
height at least $2$, then ${\rm sgn}(I)=I$ (Proposition~\ref{reviewer1}), 
and if $I=(x^a)$ is a
principal ideal, then ${\rm sgn}(A)$ is $(0,\ldots,0)^\top$ and ${\rm
sgn}(I)=R$.

The signature defines an equivalence relation on the semigroup 
$\mathcal{M}_n(R)$ of monomial ideals,  given
by, $I\sim J$ if and only if ${\rm sgn}(I)={\rm sgn}(J)$. The
equivalence classes are called the \textit{signature classes} of 
$\mathcal{M}_n(R)$. The signature class of ${\rm sgn}(I)$ is denoted by $[{\rm
sgn}(I)]$. If $I=R$, then $[R]$ is the set of 
 principal ideals of $\mathcal{M}_n(R)$. The fact that ${\rm
sgn}({\rm sgn}(I))={\rm sgn}(I)$  for any $I\in\mathcal{M}_n(R)$ simply
means that $I\in [{\rm sgn}(I)]$ (Lemma~\ref{sgn2=sgn}). 
Thus, the signature gives simple
representatives for the signature classes that may
be easier to handle computationally since the degrees of the monomials
of $G({\rm sgn}(I))$ are smaller than those of the monomials of
$G(I)$. 

Given $q\geq 1$ and $n\geq 2$, although the set of monomial ideals of $\mathcal{M}_n(R)$
minimally generated by 
$q$ monomials is 
infinite (Example~\ref{q-n-infinity}), the set of signatures of these ideals is finite
\cite{moran-tesis} (Lemma~\ref{nov30-25}). In his PhD thesis Moran \cite{moran-tesis} gives an
algorithm to compute all the possible signatures for monomial ideals with $q$ minimal generators
in $n$ variables. 

We are able to identify properties and algebraic invariants
that are the same for all ideals in a signature class of 
$\mathcal{M}_n(R)$ or for all
ideals of height at least $2$ in a signature class of $\mathcal{M}_n(R)$.  

The algebraic invariants of monomial ideals that we consider in this work are the depth,
projective dimension, Krull dimension, regularity, and v-number
\cite{min-dis-generalized,Eisen}. Let $I$ be a graded ideal of $R$. 
The depth and Krull dimension of $R/I$, and the $a$-invariant,
regularity, and 
depth of $R/I$ are
related by the inequalities \cite{Mats}, \cite[Corollary~B.4.1]{Vas1}:
\begin{equation*}
{\rm depth}(R/I)\leq\dim(R/I)\ \mbox{ and }\ a(R/I)\leq{\rm 
reg}(R/I)-{\rm depth}(R/I),
\end{equation*}
respectively, with equality everywhere if and only if $R/I$ is Cohen--Macaulay. 
The depth and projective dimension are related by the  
Auslander--Buchsbaum formula \cite{BHer}: 
\begin{equation*}\label{auslander-buchsbaum-formula}
{\rm pd}_R(R/I)+{\rm depth}(R/I)=\dim(R).
\end{equation*}
\quad The ring $R/I$ is called \textit{Gorenstein} if $R/I$ is a Cohen--Macaulay
ring of type $1$ \cite{BHer}. We say that $I$ is
\textit{Cohen--Macaulay} (resp. \textit{Gorenstein}) if $R/I$ is Cohen--Macaulay (resp.
Gorenstein). 

For monomial ideals, Mart\'\i nez-Bernal et al. \cite[Theorems~3.1 and
3.7]{depth-monomial} proved that, one can use the polarization technique due to
Fr\"oberg \cite[p.~203]{monalg-rev} to show that the depth of $R/I$ is locally
non-decreasing at each variable $x_i$ when lowering the top degree of
$x_i$ that occurs in $G(I)$,
and that one can also control the regularity. In
Section~\ref{tight-matrix-section}, we introduce the \textit{shift} and 
the \textit{weighted partial polarization} operations:
$$
I\longmapsto I_{\rm sft}\ \mbox{ and }\ I\longmapsto I_{\rm pol},
$$
respectively, using a new weighted variable $x_0$ and letting
$S=R[x_0]$. A recursive applications of these
operations gives the signature of $I$ and the full weighted
polarization of $I$, respectively (Proposition~\ref{sep22-25}, 
Example~\ref{signature-example}). This allows us to show that the
signature preserves the depth and lowers the regularity of $R/I$ 
(Theorems~\ref{depth=sgn} and \ref{reg-sgn}), and
to compare the minimal free resolutions of $R/I$ and $R/{\rm sgn}(I)$
with that of $S/I_{\rm pol}$ 
(Eqs.~\eqref{dec25-25}--\eqref{jan1-26-1} and
Remark~\ref{jan1-26-remark}). We compare the irreducible
decompositions of $I$ and ${\rm sgn}(I)$
(Theorem~\ref{irred-i-sgni-theo}).

We come to our first main result. As is seen below, parts (a) and (d)
hold more generally 
for non-principal monomial ideals.

\medskip

\noindent{\bf Theorem~\ref{sgn-invariants}}\textit{\ Let $I$ be a
monomial ideal of $R$ of height at least 
$2$. The following hold. 
\begin{enumerate} 
\item[\rm(a)] ${\rm depth}(R/I)={\rm depth}(R/{\rm sgn}(I))$.
\item[\rm(b)] ${\rm dim}(R/I)={\rm dim}(R/{\rm sgn}(I))$.
\item[\rm(c)] $R/I$ is Cohen--Macaulay $($resp. Gorenstein$)$ if and only 
if $R/{\rm sgn}(I)$ is Cohen--Macaulay $($resp. Gorenstein$)$.
\item[\rm(d)] ${\rm reg}(R/I)\geq{\rm reg}(R/{\rm sgn}(I))$.
\end{enumerate}
}

As a consequence, we recover two results of Gimenez at al. 
\cite{cm-oriented-trees}. One of them, comparing the Cohen--Macaulay property  
of squarefree monomial ideals obtained by assigning positive weights to the 
variables of the ring $R$, and the other comparing the Cohen--Macaulay
property of certain weighted oriented graphs (Corollaries~\ref{coro1} and
\ref{oct20-17}). 

We come to our main result on the depth of quotient rings of monomial
ideals that gives a positive answer to a recent
conjecture of Ibarguen et al.
\cite[Conjecture~5.1]{signature}. 

\medskip

\noindent \textbf{Theorem~\ref{depth=sgn}}\textit{
Let $I$ be a non-principal monomial ideal of $R$. Then, 
$$
{\rm depth}(R/I)={\rm depth}(R/{\rm sgn}(I)).
$$
}

The signature of a monomial ideal $I$ lowers the degrees of the
minimal generators of $I$. The following result shows that 
the signature of $I$ lowers the regularity of $I$.  

\noindent \textbf{Theorem~\ref{reg-sgn}}\textit{
Let $I$ be a non-principal monomial ideal of $R$. Then, 
$$
{\rm reg}(R/I)\geq{\rm reg}(R/{\rm sgn}(I)).
$$
}

The next invariant that we consider is the v-number of $I$, that is closely
related to the associated primes of $I$.   
A prime ideal $\mathfrak{p}$ of $R$ is an \textit{associated prime} 
of a monomial ideal $I$ if ${\mathfrak p}=(I\colon x^a)$ for 
some $x^a\in R$, where $(I\colon x^a):=\{g\in R\mid gx^a\in I\}$
is a \textit{colon ideal}. The set of associated primes of $I$ is denoted by
${\rm Ass}(I)$. The v-{\em number} of $I$, denoted ${\rm
v}(I)$, is the following invariant of
$I$ that was introduced for graded ideals by Cooper et al. 
\cite[Definition~4.1]{min-dis-generalized}:
\begin{equation*}
{\rm v}(I):=\min\{d\geq 0 \mid\exists\, x^a\mbox{ of degree }d
\mbox{ and }\mathfrak{p} \in {\rm Ass}(I) \mbox{ with } (I\colon x^a)
=\mathfrak{p}\}.
\end{equation*} 
\quad 
The notion of {\rm v-number} is relatively new, we refer to
\cite{civan,Ficarra,
v-number,Terai-v-number,
Saha,
saha-sengupta,wolmer-survey} and
references therein for some of the main results on the v-number of
monomial ideals. 

To compute the associated primes and {\rm v}-number of monomial
ideals, 
we use \textit{Macaulay}$2$ \cite{mac2}
together with the algorithms given in \cite[Theorem 1]{im-vnumber} and
\cite[Proposition~4.2]{min-dis-generalized}. 

For using below recall that an ideal $I$ is 
called \textit{unmixed} if ${\rm ht}(I)={\rm ht}(\mathfrak{p})$ for all
$\mathfrak{p}\in{\rm Ass}(I)$, where ${\rm ht}(I)$ denotes the 
\textit{height} or \textit{codimension} of $I$.   

Our main result on v-numbers shows that the v-number of ${\rm sgn}(I)$ is a
lower bound for the v-number of $I$. To prove this, we show that 
${\rm Ass}(I)={\rm Ass}({\rm sgn}(I))$ when ${\rm ht}(I)\geq 2$,  
by proving that given an associated prime $\mathfrak{p}=(I\colon x^a)$
of $I$, we can
find $x^b$ in terms of $x^a$ and $I_{\rm sft}$ such that 
$\mathfrak{p}=(I_{\rm sft}\colon x^b)$ and $\deg(x^b)\leq\deg(x^a)$,
and vice versa (Corollary~\ref{ass-v-number-coro},
Example~\ref{example-ass-ij}). 

\medskip

\noindent \textbf{Theorem~\ref{ass-v-number}}\textit{
Let $I$ be a monomial ideal of height $\geq 2$ and let ${\rm v}(I)$ be
its ${\rm v}$-number. Then
\begin{enumerate}
\item[\rm(i)] ${\rm Ass}(I)={\rm Ass}({\rm sgn}(I))$.
\item[\rm(ii)] ${\rm v}({\rm sgn}(I))\leq {\rm v}(I)$.
\item[\rm(iii)] $I$ is unmixed if and only if ${\rm sgn}(I)$ is
unmixed. 
\end{enumerate}
}

In Section~\ref{section-examples}, we give a list of examples 
to illustrate and complement some of our results. The main example
shows the recursive use of the shift and weighted partial polarization operations
to obtain the signature and the full weighted polarization of a
monomial ideal 
(Example~\ref{signature-example}). 

Appendix~\ref{AppendixA} includes a procedure for 
\textit{Macaulay}$2$ \cite{mac2} to compute the signature of a
monomial ideal (Procedure~\ref{procedure-signature}), and a procedure
to examine given families of monomial ideal by computing 
their signature ideals and determining which signature ideals are
Cohen--Macaulay or Gorenstein (Procedure~\ref{procedure-list}). Then,
in Appendix~\ref{AppendixB}, 
using the list of signatures matrices of
\cite[Chapter~4]{moran-tesis}, we give the list of all signature 
Cohen--Macaulay monomial ideals within a certain
range.

For all unexplained
terminology and additional information  we refer to
\cite{herzog-hibi-book,monalg-rev}
(for monomial ideals), and
\cite{BHer,Mats} (for commutative rings).

\section{The depth and regularity of a monomial ideal and its
signature}\label{tight-matrix-section}
Let $I$ be a monomial ideal of $R$ and let 
$A$ be its incidence matrix. In this section, we introduce the shift operation 
$I\mapsto I_{\rm sft}$ and the weighted partial polarization
$I\mapsto I_{\rm pol}$. The shift operation eliminates ``gaps'' in
$I$ using the weighted partial polarization. 
A recursive applications of these
operations give the signature of $I$ and the full weighted
polarization of $I$, respectively. This allows us to show that the
signature preserves the depth and lowers the regularity of $R/I$, and
to compare the minimal free resolutions of $I$ and ${\rm sgn}(I)$. 
For these operations to work, we need to assume that the height of
$I$ is at least $2$. 
To avoid repetitions we continue to use the notation introduced in
Section~\ref{intro-section}.

\subsection{Procedure to transform $A$ into a tight
matrix}\label{tight-matrix-subsection} 
Let $I$ be a monomial ideal of $R$ of height at least $2$ and let 
$A$ be its incidence matrix. Then, each row of $A$ has at least one
zero entry. We give a recursive procedure that transforms $A$ into a tight matrix whose corresponding
monomial ideal is the signature of $I$.  

If $A$ is not a tight matrix pick any row of $A$ that is not tight. For
simplicity of notation assume that the first row of $A$ is not tight.
Then, after listing the monomials of $G(I)$ in ascending order with respect to
the powers of $x_1$, there are integers $p\geq 0$, $q_1\leq\cdots\leq q_s$, and
$k_0,\ldots,k_p$, such that 
$q_1-p\geq 2$, $k_i\geq 1$ for $i=0,\ldots,p$, and we can write 
\begin{align}
G(I)=&\big\{x_1^jx^{\gamma_{i,j}}\mid 0\leq j\leq p,\, 1\leq i\leq k_j,\,
x_1\notin{\rm
supp}(x^{\gamma_{i,j}})\}\textstyle\bigcup\{x_1^{q_i}x^{\epsilon_i}\mid 1\leq i\leq
s,\, x_1\notin{\rm
supp}(x^{\epsilon_i})\big\}\nonumber\\
=&\big\{x^{\gamma_{1,0}},\ldots,x^{\gamma_{k_0,0}},
x_1x^{\gamma_{1,1}},\ldots,x_1x^{\gamma_{k_1,1}},\ldots,
x_1^px^{\gamma_{1,p}},\ldots,x_1^px^{\gamma_{k_p,p}},
x_1^{q_1}x^{\epsilon_1},\ldots,x_1^{q_s}x^{\epsilon_s}\big\}, and \nonumber\\
&G(I)=\big[\textstyle\bigcup_{j=0}^p\big\{x_1^jx^{\gamma_{1,j}},\ldots,x_1^jx^{\gamma_{i,j}},
\ldots,x_1^jx^{\gamma_{k_j,j}}\big\}
\big]
\bigcup\big\{x_1^{q_1}x^{\epsilon_1},\ldots,x_1^{q_s}x^{\epsilon_s}\big\},\label{i-def}
\end{align}
where $x_1\notin{\rm supp}(x^{\gamma_{i,j}})$ for all $i,j$ and
$x_1\notin{\rm supp}(x^{\epsilon_i})$ for all $i$. The existence of
$k_0,\ldots,k_p$ such that $k_i\geq 1$ for all $i$ is
guaranteed by the assumption that 
${\rm ht}(I)\geq 2$. Recalling that 
$q_1-p\geq 2$, we say $I$ has
a \textit{gap} with respect to $x_1$ occurring at 
$x_1^px^{\gamma_{k_p,p}},x_1^{q_1}x^{\epsilon_1}$. The idea is to
remove this gap, using the following procedure, without affecting the depth and the
dimension of $R/I$. If $x^a\in G(I)$ and $x_k\in{\rm supp}(x^a)$, we also say that $I$ has 
a \textit{gap} in $x^a$ at $x_k$ whenever there does not exist 
$x^b\in G(I)$ such that $b_k=a_k-1$.

We now introduce a \textit{local degree reduction} ideal that we call a
\textit{weighted partial polarization} of $I$. 
Consider the monomial ideal $I_{\rm pol}$ obtained from $G(I)$ by adding a new variable
$x_0$ to $R$ and making the substitution $x_1^{q_1-p}\rightarrow x_0$ in 
the monomials
$x_1^{q_1}x^{\epsilon_1},\ldots,x_1^{q_s}x^{\epsilon_s}$, that is, 
$I_{\rm pol}$ is the monomial ideal of $R[x_0]$ given by 
\begin{equation}\label{pol-def}
I_{\rm pol}:=\Big(\big[\textstyle\bigcup_{j=0}^p\big\{x_1^jx^{\gamma_{1,j}},
\ldots,x_1^jx^{\gamma_{i,j}},\ldots, x_1^jx^{\gamma_{k_j,j}}\big\}\big]
\bigcup\big\{x_0x_1^{q_i-q_1+p}x^{\epsilon_i}\big\}_{i=1}^s\Big).
\end{equation}
\quad It is not hard to see that the set of monomials generating $I_{\rm pol}$ is 
the minimal set of generators $G(I_{\rm pol})$ of $I_{\rm pol}$.
Whenever needed, we may assume that $x_0-x_1^{q_1-p}$ is homogeneous by
letting $\deg(x_0)=q_1-p$. To eliminate the gap 
``$x_1^px^{\gamma_{k_p,p}},x_1^{q_1}x^{\epsilon_1}$'' relative to
$x_1$, we consider the ideal $I_{\rm sft}$ of $R$ obtained from $I_{\rm pol}$ by making 
$x_0=x_1$:  
\begin{align}
I_{\rm sft}:&=\Big(\big[\textstyle\bigcup_{j=0}^p\big\{x_1^jx^{\gamma_{1,j}},
\ldots,x_1^jx^{\gamma_{i,j}},\ldots, x_1^jx^{\gamma_{k_j,j}}\big\}\big]
\bigcup\big\{x_1^{q_i-q_1+p+1}x^{\epsilon_i}\big\}_{i=1}^s\Big)\nonumber\\
&=\Big(\big[\textstyle\bigcup_{j=0}^p\big\{x_1^jx^{\gamma_{1,j}},
\ldots,x_1^jx^{\gamma_{i,j}},\ldots, x_1^jx^{\gamma_{k_j,j}}\big\}\big]
\bigcup\big\{x_1^{\ell_i}x^{\epsilon_i}\big\}_{i=1}^s\Big),\label{sft-def}
\end{align}
where $\ell_i=q_i-q_1+p+1$. We call the ideal $I_{\rm sft}$ a \textit{shift} of $I$.  
Note that $I\subset I_{\rm sft}$. The next lemma shows that if $I_{\rm sft}$ has a gap, we can
successively apply the procedure above to obtain a gap free ideal
whose incidence matrix has a tight first row. Then, we apply the same
procedure to any other row to obtain a tight matrix.

\begin{lemma}\label{lemma-termination-procedure} The minimal set of generators of $I_{\rm sft}$ is 
\begin{align*}
G(I_{\rm sft})=&\big[\textstyle\bigcup_{j=0}^p\big\{x_1^jx^{\gamma_{1,j}},
\ldots,x_1^jx^{\gamma_{i,j}},\ldots, x_1^jx^{\gamma_{k_j,j}}\big\}\big]
\bigcup\big\{x_1^{q_i-q_1+p+1}x^{\epsilon_i}\big\}_{i=1}^s\\
&\quad\quad=\big[\textstyle\bigcup_{j=0}^{p+1}\big\{x_1^jx^{\gamma_{1,j}},
\ldots,x_1^jx^{\gamma_{i,j}},\ldots, x_1^jx^{\gamma_{k_j,j}}\big\}\big]
\bigcup\big\{x_1^{q_i'}x^{\epsilon_i'}\big\}_{i=1}^{s'},
\end{align*}
where $p+3\leq q_1'\leq\cdots\leq q_{s'}'$ if $I_{\rm sft}$ has a gap at 
``$x_1^{p+1}x^{\gamma_{k_{p+1},p+1}},x_1^{q_1'}x^{\epsilon_1'}$'' and
$\{x_1^{q_i'}x^{\epsilon_i'}\big\}_{i=1}^{s'}$ is empty if $I_{\rm sft}$ has no
gaps. 
\end{lemma}
\begin{proof} We need only show that the monomials that generate $I_{\rm sft}$
form an anti-chain, that is, they are non-comparable with respect to
the poset $\mathbb{M}_n$ of monomials of $R$ under divisibility. 

Case (A): $x_1^jx^{\gamma_{i,j}}=(x_1^{q_k-q_1+p+1}x^{\epsilon_k})x^c$
for some $x^c$.  
Then $j\geq q_k-q_1+p+1\geq p+1$, a contradiction since $0\leq j\leq p$.

Case (B):
$x_1^{q_k-q_1+p+1}x^{\epsilon_k}=(x_1^jx^{\gamma_{i,j}})x^c$ for some
$x^c$. Then,
$x_1^{q_k}x^{\epsilon_k}=(x_1^jx^{\gamma_{i,j}})x^c(x_1^{q_1-p-1})$, a
contradiction since $G(I)$ is an anti-chain of $\mathbb{M}_n$. 

The remaining cases are also easy to show using that $G(I)$ is an
anti-chain of $\mathbb{M}_n$.  
\end{proof}

\begin{proposition}\label{carlitos-vila} Let $I$ be a monomial ideal of height at 
least $2$ and let $A$ be its incidence matrix. By recursively 
applying the shift operation $I\mapsto I_{\rm sft}$ to all rows of
$A$, we
obtain the signature ${\rm sgn}(I)$ of $I$ and the signature ${\rm
sgn}(A)$ of $A$. 
\end{proposition}

\begin{proof} By the procedure above and
Lemma~\ref{lemma-termination-procedure}, we can recursively eliminate all gaps from
$I$ and since each row of $A$ has at least one zero entry the procedure has
to end precisely when we reach ${\rm sgn}(I)$, i.e., when all rows are tight.
\end{proof}

\subsection{The signature invariants}

On the ring theory side, we show that the shift and signature
operations $I\mapsto
I_{\rm sft}(I)$ and $I\mapsto {\rm sgn}(I)$ preserve the depth, the
Krull dimension, and some of the algebraic properties of $R/I$.

\begin{definition}\rm An ideal obtained by recursively applying the 
operation $I\mapsto I_{\rm pol}$ is called a \textit{weighted
polarization} of $I$. A \textit{weighted polarization} is
\textit{full} if the corresponding recursive operation $I\mapsto
I_{\rm sft}$ ends
at the signature of $I$.
\end{definition}

The next lemma is not hard to prove.

\begin{lemma}\label{quotient-I-xi}
Let $I\subset R$ be a monomial ideal such that $G(I)$ has the form
$$
G(I)=\{x^{\gamma_1},\ldots,x^{\gamma_r},x_1^{d_1}x^{\delta_1},\ldots,
x_1^{d_t}x^{\delta_t}\},
$$
where $x_1\notin{\rm supp}(x^{\gamma_i})$, $x_1\notin{\rm
supp}(x^{\delta_j})$ for all $i,j$, $r\geq 1$, $t\geq 1$, and
$1=d_1\leq d_2\leq\cdots\leq d_t$. Then, the minimal generating set of
$(I\colon x_1)$ is given by 
$$
G(I\colon x_1)=\{x^{\gamma_i}\mid x^{\delta_j}\mbox{ does not divides
}x^{\gamma_i}\mbox{ when
}d_j=1\}\textstyle\bigcup\{x_1^{d_i-1}x^{\delta_i}\}_{i=1}^t.
$$
\end{lemma}

\begin{lemma}\label{I-Ipol-x0} 
Let $I\subset R$ be a monomial
ideal of height at least $2$ and let $I_{\rm pol}$ be the weighted partial polarization of
Eq.~\eqref{pol-def}. Then, there is a ring isomorphism
$$
\overline{\varphi}\colon R/I\longrightarrow R[x_0]/(I_{\rm
pol},x_0-x_1^{q_1-p}), \quad \overline{f}\longmapsto\widetilde{f},
$$
that is also a graded isomorphism of $R$-modules, where
$\deg(x_0)=q_1-p$. In particular,    
$${\rm reg}(R/I)={\rm reg}_R(R[x_0]/(I_{\rm pol},x_0-x_1^{q_1-p})).$$
\end{lemma}

\begin{proof} There is a graded homomorphism of rings 
$\varphi\colon R\longrightarrow R[x_0]/(I_{\rm
pol},x_0-x_1^{q_1-p})$, $f\longmapsto\widetilde{f}$, 
that is also a graded homomorphism of $R$-modules. This map is onto
since $\varphi(x_1^{q_1-p})=\widetilde{x_0}$. We let $P:=(I_{\rm
pol},x_0-x_1^{q_1-p})$. Note that $I\subset{\rm ker}(\varphi)$. Indeed, 
take $f\in G(I)$. If $f=x_1^jx^{\gamma_{i,j}}$, then $f\in I_{\rm
pol}$ and $\varphi(f)=\widetilde{0}$. If 
$f=x_1^{q_i}x^{\epsilon_i}$, then from the equality 
$$
x_1^{q_i}x^{\epsilon_i}-x_0x_1^{q_i-q_1+p}x^{\epsilon_i}=
x_1^{q_i-q_1+p}x^{\epsilon_i}(x_1^{q_1-p}-x_0)\in P,
$$
one has 
$\varphi(f)=\widetilde{f}=x_0x_1^{q_i-q_1+p}+P=\widetilde{0}$. Thus, 
$I\subset{\rm ker}(\varphi)$, and consequently $\varphi$ 
induces an epimorphism $\overline{\varphi}\colon R/I\rightarrow
R[x_0]/P$ such that 
$\overline{\varphi}(\overline{f})=\widetilde{f}$. To show that
$\overline{\varphi}$ is an isomorphism, we need
only show that ${\rm ker}(\varphi)\subset I$. Take $f\in{\rm
ker}(\varphi)$. Then, we can write
$f=h_1(x_0-x_1^{q_1-p})+h_2x_1^jx^{\gamma_{i,j}}$ or 
$f=h_3(x_0-x_1^{q_1-p})+h_4(x_0x_1^{q_i-q_1+p}x^{\epsilon_i})$. Hence, making
$x_0=x_1^{q_1-p}$, gives $f\in I$. 
\end{proof}

\begin{remark}\label{I-Ipol-x0-rem} 
Letting $\deg(x_0)=q_1-p$. There is a ring isomorphism
$$
\phi\colon R\longrightarrow R[x_0]/(x_0-x_1^{q_1-p}), \quad 
f\longmapsto\overline{f},
$$
that is also a graded isomorphism of $R$-modules, 
$0={\rm reg}(R)={\rm reg}_R(R[x_0]/(x_0-x_1^{q_1-p}))$, and the
regularity of $R[x_0]/(x_0-x_1^{q_1-p})$, as an $R[x_0]$-module, is 
$q_1-p-1$.
\end{remark}

\begin{lemma}\label{depth-dim-inv} Let $I\subset R$ be a monomial
ideal of height at least $2$ and let $I_{\rm pol}$ and $I_{\rm sft}$ be the ideals of
Eqs.~\eqref{pol-def} and \eqref{sft-def}, respectively. The following hold. 
\begin{enumerate} 
\item[\rm(a)] The pure binomials $x_0-x_1^{q_1-p}$ and $x_0-x_1$ are
regular on $R[x_0]/I_{\rm pol}$, that is, they are non-zero divisors of
$R[x_0]/I_{\rm pol}$.
\item[\rm(b)] $(R[x_0]/I_{\rm pol})/(x_0-x_1^{q_1-p})=R/I$ and
$(R[x_0]/I_{\rm pol})/(x_0-x_1)=R/I_{\rm sft}$.
\item[\rm(c)] ${\rm depth}(R/I)={\rm depth}(R/I_{\rm sft})$ and
$\dim(R/I)=\dim(R/I_{\rm sft})$.
\item[\rm(d)] $R/I$ is Cohen--Macaulay if and only if $R/I_{\rm sft}$ 
is Cohen--Macaulay.
\item[\rm(e)] $R/I$ is Gorenstein if and only if $R/I_{\rm sft}$ is
Gorenstein.
\item[\rm(f)] If $p\geq 1$, then $(I_{\rm sft}\colon x_1)=(I\colon
x_1)_{\rm sft}$. If $p=0$, then $(I\colon x_1^{q_1-1})=I_{\rm sft}$.
\end{enumerate}
\end{lemma}

\begin{proof} (a) We argue by contradiction assuming that
$x_0-x_1^{q_1-p}$ (resp. $x_0-x_1$) is a zero divisor of $R[X_0]/I_{\rm pol}$.
Then, $x_0-x_1^{q_1-p}$ (resp. $x_0-x_1$) belongs to an associated 
prime ${\mathfrak p}$ of $I_{\rm pol}$ and 
we can write ${\mathfrak p}=(I_{\rm pol}\colon h)$, for some monomial $h=x^a$
in $R[x_0]$. Note that $h\notin I_{\rm pol}$. Since 
${\mathfrak p}$ is a face ideal generated by a set of variables of
$R[x_0]$ and $x_0-x_1^{q_1-p}\in\mathfrak{p}$, 
it follows readily that $x_i\in\mathfrak{p}$ for $i=0,1$, and one has
$x_ih\in I_{\rm pol}$ for $i=0,1$. 
Therefore, from Eq.~\eqref{pol-def}, we can
write
$$
\begin{array}{ccc}
x_0h=\begin{cases}
{\rm(i)}\ x_1^jx^{\gamma_{i,j}}x^{c_1},\, 0\leq j\leq p,\mbox{ or}& \\
{\rm(ii)}\ x_0x_1^{q_i-q_1+p}x^{\epsilon_i}x^{c_2}, &
\end{cases}
 &\mbox{ and }\ & 
x_1h=\begin{cases}
{\rm(iii)}\ x_1^{j'}x^{\gamma_{i',j'}}x^{c_3},\, 0\leq j'\leq p,\mbox{ or}& \\
{\rm(iv)}\ x_0x_1^{q_k-q_1+p}x^{\epsilon_k}x^{c_4}, &
\end{cases}
\end{array}
$$
for some $x^{c_i}\in R[x_0]$, $i=1,\ldots,4$. Clearly (i) cannot occur
since $h\notin I_{\rm pol}$. Then (ii) holds and
$h=x_1^{q_i-q_1+p}x^{\epsilon_i}x^{c_2}$. Note also that
$x_0\notin{\rm supp}(x^{c_2})$, because otherwise $h\in I_{\rm pol}$, a
contradiction. There are two cases to consider.

Case (A): Assume (ii) and (iii) hold. Then
\begin{equation}\label{sep2-25}
x_1^{j'}x^{\gamma_{i',j'}}x^{c_3}\stackrel{\rm(iii)}{=}x_1h
\stackrel{\rm(ii)}{=}x_1(x_1^{q_i-q_1+p}x^{\epsilon_i}x^{c_2}),
\end{equation}
where $0\leq j'\leq p$. 
If $x_1\in {\rm supp}(x^{c_3})$, then from the first equality of
Eq.~\eqref{sep2-25} $h\in I_{\rm pol}$, a
contradiction. Thus, $x_1\notin {\rm supp}(x^{c_3})$. 
Hence, comparing powers of $x_1$ in Eq.~\eqref{sep2-25}, we get 
$j'\geq 1+q_i-q_1+p\geq p+1$, a contradiction.

Case (B): Assume (ii) and (iv) hold. Then
$$
x_0x_1^{q_k-q_1+p}x^{\epsilon_k}x^{c_4}\stackrel{\rm(iv)}{=}x_1h   
\stackrel{\rm(ii)}{=}x_1(x_1^{q_i-q_1+p}x^{\epsilon_i}x^{c_2}),
$$
and consequently $x_0$ divides $x^{c_2}$. Recalling that 
$x_0\notin{\rm supp}(x^{c_2})$, we get a contradiction.

(b) Letting $f=x_0-x_1^{q_1-p}$ and $g=x_0-x_1$, one has
\begin{align*}
&(R[x_0]/I_{\rm pol})/(f)=(R[x_0]/I_{\rm pol})/((I_{\rm pol}+(f))/I_{\rm pol})=R[x_0]/(I_{\rm pol}+(f))
=R/I,\\
&(R[x_0]/I_{\rm pol})/(g)=(R[x_0]/I_{\rm pol})/((I_{\rm pol}+(g))/I_{\rm pol})=R[x_0]/(I_{\rm pol}+(g))
=R/I_{\rm sft}.
\end{align*}

(c) Letting $M=R[x_0]/I_{\rm pol}$, $f=x_0-x_1^{q_1-p}$, and $g=x_0-x_1$, by
part (a), $f$ and $g$ are regular on $M$. Then, by 
\cite[Lemma~2.3.10]{monalg-rev}, we get 
\begin{align*}
&{\rm depth}(M/fM)={\rm depth}(M)-1\mbox{ and }\dim(M/fM)=\dim(M)-1,\\
&{\rm depth}(M/gM)={\rm depth}(M)-1\mbox{ and }\dim(M/gM)=\dim(M)-1.
\end{align*}
\quad Therefore, by part (b), we obtain
\begin{align*}
&{\rm depth}(R/I)={\rm depth}(M/fM)={\rm depth}(M/gM)={\rm
depth}(R/I_{\rm sft}),\\
&\dim(R/I)=\dim(M/fM)=\dim(M/gM)=\dim(R/I_{\rm sft}).
\end{align*}
\quad (d) It follows at once from part (c) by recalling 
that $R/I$ is Cohen--Macaulay if and only if ${\rm
depth}(R/I)=\dim(R/I)$ \cite[Definition 2.3.8]{monalg-rev}.

(e) Letting $\deg(x_0)=q_1-p$, we regard $f=x_0-x_1^{q_1-p}$ as a
homogeneous element of $R[x_0]$. Then, the result follows from (a) and
(b), and 
the fact that $M$ is Gorenstein if and only if $M/fM=R/I$ (resp.
$M/gM=R/I_{\rm sft}$) 
is Gorenstein. This fact follows from the graded version 
of \cite[Proposition~3.1.19(b)]{BHer}.

(f) Assuming that $p\geq 1$, we now show the equality $(I_{\rm sft}\colon x_1)=(I\colon
x_1)_{\rm sft}$. By Eq.~\eqref{i-def} and Lemma~\ref{quotient-I-xi}, we obtain
\begin{align}
G(I)&=\big[\textstyle\bigcup_{j=0}^p\big\{x_1^jx^{\gamma_{1,j}},\ldots,x_1^jx^{\gamma_{i,j}},
\ldots,x_1^jx^{\gamma_{k_j,j}}\big\}
\big]
\bigcup\big\{x_1^{q_i}x^{\epsilon_i}\big\}_{i=1}^s,\\
G(I\colon
x_1)&=\Gamma\textstyle\bigcup\big[\textstyle\bigcup_{\ell=0}^{p-1}
\big\{x_1^\ell x^{\gamma_{1,\ell+1}},\ldots,x_1^\ell 
x^{\gamma_{i,\ell+1}},\ldots,x_1^\ell x^{\gamma_{k_{\ell+1},\ell+1}}\big\}
\big]\bigcup\big\{x_1^{q_i-1}x^{\epsilon_i}\big\}_{i=1}^{s},\label{gcolonx1}
\end{align}
where $\Gamma=\{x^{\gamma_{i,0}}\mid x^{\gamma_{j,1}}\mbox{ does not divides
}x^{\gamma_{i,0}}\mbox{ for all }1\leq j\leq k_1
\}$. Therefore,
\begin{equation*}
(I\colon x_1)_{\rm sft}=\big(\Gamma\textstyle\bigcup\big[\textstyle\bigcup_{\ell=0}^{p-1}
\big\{x_1^\ell x^{\gamma_{1,\ell+1}},\ldots,x_1^\ell 
x^{\gamma_{i,\ell+1}},\ldots,x_1^\ell x^{\gamma_{k_{\ell+1},\ell+1}}\big\}
\big]\bigcup\big\{x_1^{q_i-q_1+p}x^{\epsilon_i}\big\}_{i=1}^{s}\big).
\end{equation*}
\quad On the other hand, 
by Lemmas~\ref{lemma-termination-procedure} and \ref{quotient-I-xi}, we obtain 
\begin{align}
G(I_{\rm sft})&=\big[\textstyle\bigcup_{j=0}^p\big\{x_1^jx^{\gamma_{1,j}},
\ldots,x_1^jx^{\gamma_{i,j}},\ldots, x_1^jx^{\gamma_{k_j,j}}\big\}\big]
\bigcup\big\{x_1^{q_i-q_1+p+1}x^{\epsilon_i}\big\}_{i=1}^s\nonumber,\\
(I_{\rm sft}\colon
x_1)&=\big(\Gamma\textstyle\bigcup\{x^{\gamma_{i,1}}\}_{i=1}^{k_1}\textstyle\bigcup\big[\textstyle\bigcup_{j=2}^{p}
\big\{x_1^{j-1} x^{\gamma_{1,j}},\ldots,
x_1^{j-1} x^{\gamma_{k_{j},j}}\big\}
\big]\bigcup\big\{x_1^{q_i-q_1+p}x^{\epsilon_i}\big\}_{i=1}^{s}\big).\nonumber
\end{align}
Hence, making $j-1=\ell$, it follows that $(I\colon x_1)_{\rm sft}$ is
equal to $(I_{\rm sft}\colon x_1)$.

Assuming that $p=0$, we now show the equality $(I\colon
x_1^{q_1-1})=I_{\rm sft}$. By Eq.~\eqref{i-def} and
Lemma~\ref{lemma-termination-procedure}, one has the equalities
\begin{align*}
G(I)&=\{x^{\gamma_{1,0}},\ldots,x^{\gamma_{i,0}},
\ldots,x^{\gamma_{k_0,0}}\}
\textstyle\bigcup\big\{x_1^{q_i}x^{\epsilon_i}\big\}_{i=1}^s,\\
G(I_{\rm sft})&=\{x^{\gamma_{1,0}},\ldots,x^{\gamma_{i,0}},
\ldots,x^{\gamma_{k_0,0}}\}
\textstyle\bigcup\big
\{x_1^{q_i-q_1+1}x^{\epsilon_i}\big\}_{i=1}^s,
\end{align*}
where $q_1-p=q_1\geq 2$. Letting $\ell_i=q_i-q_1+1$ for
$i=1,\ldots,s$, the inclusion $(I\colon
x_1^{q_1-1})\supset I_{\rm sft}$ is clear because $p=0$ and
$x_1^{q_1-1}(x^{\ell_i}x^{\epsilon_i})=x^{q_i}x^{\epsilon_i}\in I$ 
for $i=1,\ldots,s$. To show the inclusion $(I\colon
x_1^{q_1-1})\subset I_{\rm sft}$, take $x^a\in(I\colon
x_1^{q_1-1})$. If $x^ax_1^{q_1-1}=x^{\gamma_{k,0}}x^\theta$, then 
$x^a=x^{\gamma_{k,0}}x^{\theta_1}\in I_{\rm sft}$ because $x_1\notin{\rm
supp}(x^{\gamma_{k,0}})$. If
$x^ax_1^{q_1-1}=x_1^{q_i}x^{\epsilon_i}x^\theta$, then 
$x^a=x_1^{q_i-q_1+1}x^{\epsilon_i}x^\theta\in I_{\rm sft}$.
\end{proof}

The following result explains the role of a full weighted 
polarization to link a monomial ideal with its signature. 

\begin{proposition}\label{sep22-25} 
Let $I\subset R=K[x_1,\ldots,x_n]$ be a monomial ideal of height at
least two. Then, there is a polynomial ring $S=R[z_1,\ldots,z_r]$, an 
ideal $I_{\rm pol}\subset S$, a ring $M=S/I_{\rm pol}$, where $I_{\rm
pol}$ is a full weighted polarization of $I$, and two $S$-regular sequence 
$\underline{f}$ and $\underline{g}$ such that 
\begin{enumerate}
\item[\rm(a)] $\underline{f}=\{z_i-x_{j_i}^{d_i}\}_{i=1}^r$ and 
$\underline{g}=\{z_i-x_{j_i}\}_{i=1}^r$, where $d_i\geq 2$ for all
$i$,
\item[\rm(b)] $\underline{f}$ and $\underline{g}$ are $M$-regular
sequences, 
\item[\rm(c)] $M/(\underline{f})=R/I$ and $M/(\underline{g})=R/{\rm
sgn}(I)$, 
\item[\rm(d)] ${\rm ht}(I)={\rm ht}(I_{\rm pol})={\rm ht}({\rm sgn}(I))$,
and
\item[\rm(e)] $\dim(M)=\dim(R/I)+r$ and ${\rm depth}(M)={\rm
depth}(R/I)+r$.
\end{enumerate}
\end{proposition}
\begin{proof} This follows by recursively applying the 
operation $I\mapsto I_{\rm pol}$, that we use in the 
proof of Lemma~\ref{depth-dim-inv}, until the corresponding operation
$I\mapsto I_{\rm sft}$ reaches ${\rm sgn}(I)$.
\end{proof}

Let $I$ be a graded ideal of $R$ 
and let ${\mathbb F}_\star$ be the minimal graded
free resolution of $R/I$ as an 
$R$-module: 
\begin{equation*}
{\mathbb F}_\star:\ \ \ 0\rightarrow 
F_g\stackrel{}{\rightarrow} \cdots\rightarrow F_k\rightarrow\cdots 
\rightarrow F_1\stackrel{}{\rightarrow} F_0
\rightarrow R/I \rightarrow 0,
\end{equation*}
where 
$$F_0=R\mbox{ and }F_k=\bigoplus_{j}R(-j)^{b_{k,j}}\mbox{  for } k=1,\ldots,g. 
$$
\quad The {\it Castelnuovo--Mumford regularity\/} 
of $R/I$ ({\it regularity} of $R/I$ for short) is given by 
\begin{equation*}
{\rm reg}(R/I):=\max\{j-k\mid 
b_{k,j}\neq 0\},
\end{equation*} 
and the \textit{projective dimension} of $R/I$, denoted ${\rm
pd}_R(R/I)$, is equal to $g$. The quotient ring $R/I$ is \textit{Gorenstein}
if $R/I$ is Cohen--Macaulay and there is a unique $j$ such that
$b_{g,j}\neq 0$ and $b_{g,j}=1$.
 
\begin{lemma}\cite[Corollary~4.13]{eisenbud-syzygies}\label{dec26-25}
Suppose that $M$ is a finitely generated graded $S$-module. If $x$ is
a linear form in $S$ that is a non-zero-divisor on $M$, then 
${\rm reg}(M)={\rm reg}(M/xM)$. 
\end{lemma}

The following result follows adapting the proof of
\cite[Corollary~4.13]{eisenbud-syzygies}.

\begin{proposition}\label{regular-degf}
Let $M$ be a finitely generated graded $S$-module. If $f$ is
a homogeneous polynomial of $S$ of degree $\geq 1$ that is a
non-zero-divisor on $M$, 
then 
$$
{\rm reg}(M)-(\deg(f)-1)\leq{\rm reg}(M/fM)\leq{\rm reg}(M)+\deg(f)-1.
$$
\end{proposition}

\begin{proposition} Let $I$ be a monomial ideal of $R$ and let 
$I_{\rm pol}\subset S$ be a full weighted polarization of $I$. 
Then, the following hold: 
\begin{enumerate}
\item[\rm(a)] ${\rm reg}(S/I_{\rm pol})={\rm reg}(R/{\rm sgn}(I))$.
\item[\rm(b)] If $\underline{f}=\{z_i-x_{j_i}^{d_i}\}_{i=1}^r$ and
$\deg(z_i)=d_i$ 
for all $i$, 
then ${\rm reg}(R/I)={\rm reg}_R(S/(I_{\rm pol},\underline{f}))$.  
\end{enumerate}
\end{proposition}

\begin{proof} (a) Letting $M=S/I_{\rm pol}$, the equality follows from 
Proposition~\ref{sep22-25}(c) and Lemma~\ref{dec26-25}. 

(b) This part follows from Lemma~\ref{I-Ipol-x0}. 
\end{proof}

\paragraph{\it Comparing the free resolutions of $I$, ${\rm sgn}(I)$,
and $I_{\rm pol}$}

Given an $S$-module $M$ and an $S$-regular sequence $\underline{f}$
which is also an $M$-regular sequence, it is well-known 
\cite[Corollary~1.6.14]{BHer} that we have
the following two properties:

\begin{enumerate}
\item[\rm(a)] ${\rm Tor}_i(M,S/(\underline{f}))=0$ for all $i\geq 1$. 
\item[\rm(b)] If ${\mathbb F}_\star$ is a minimal free resolution of 
$M$ by free $S$-modules, then ${\mathbb
F}_\star\otimes_S(S/(\underline{f}))$ is a free resolution of 
$M\otimes_S(S/(\underline{f}))$ by free $S/(\underline{f})$-modules.
\end{enumerate}

For convenience, we explain how this works in the situation of
Lemma~\ref{depth-dim-inv} by 
comparing the minimal free resolution of $S/I_{\rm pol}$
with that of $R/I$ and $R/I_{\rm sft}$. Letting $S=R[x_0]$, $M=S/I_{\rm pol}$, $f=x_0-x_1^{q_1-p}$ and
$g=x_0-x_1$, there are short exact sequences 
\begin{align*}
&0\rightarrow S\stackrel{f}{\longrightarrow} S
\longrightarrow S/(f)\rightarrow 0,\\
&0\rightarrow S\stackrel{g}{\longrightarrow} S
\longrightarrow S/(g)\rightarrow 0.
\end{align*}
Since $f$ and $g$ are regular on $M$ tensoring the two sequences with
$M$, it is seen that we obtain short exact sequences  
\begin{align*}
&0\rightarrow M\otimes_{S}
S\stackrel{1\otimes f}{\longrightarrow}M\otimes_{S}
S\longrightarrow
M\otimes_{S}(S/(f))=M/fM\longrightarrow 0,\\
&0\rightarrow M\otimes_{S}
S\stackrel{1\otimes g}{\longrightarrow}M\otimes_{S}
S\longrightarrow
M\otimes_{S}(S/(g))=M/gM\longrightarrow 0,\\
\end{align*}
and consequently ${\rm Tor}_i(M,S/(f))=0$ and ${\rm
Tor}_i(M,S/(g))=0$ for all $i\geq 1$. Therefore, if 
\begin{equation}\label{dec25-25}
{\mathbb F}_\star:\ \ \ 0\longrightarrow F_g\stackrel{\varphi_g}{\longrightarrow}
F_{g-1}\stackrel{\varphi_{g-1}}{\ \longrightarrow}\cdots
\longrightarrow F_1\stackrel{\varphi_1}{\longrightarrow}
F_0\longrightarrow M\longrightarrow 0, 
\end{equation}
is the minimal free resolution of $M$ by free $S$-modules, then 
\begin{align}
&\ \ \ 0\longrightarrow
F_g\otimes_{S}(S/(f))\stackrel{\varphi_g\otimes{1}}{\ \ \longrightarrow}
\cdots
\stackrel{\varphi_1\otimes{1}}{\longrightarrow}
F_0\otimes_{S}(S/(f))\longrightarrow
M\otimes_{S}(S/(f))\longrightarrow 0\mbox{ and }\label{jan1-26}\\  
&\ \ \ 0\longrightarrow
F_g\otimes_{S}(S/(g))\stackrel{\varphi_g\otimes{1}}{\ \ \longrightarrow}
\cdots
\stackrel{\varphi_1\otimes{1}}{\longrightarrow}
F_0\otimes_{S}(S/(g))\longrightarrow
M\otimes_{S}(S/(g))\longrightarrow 0,\label{jan1-26-1} 
\end{align}
are the free resolution of $M/fM=R/I$ by free $S/(f)$-modules and 
 $M/gM=R/I_{\rm sft}$ by free $S/(g)$-modules, respectively.

\begin{remark}\label{jan1-26-remark} This means that the minimal free resolution of $R/I$ is
obtained from that of $M=R[x_0]/I_{\rm pol}$ by making
the substitution $x_0\rightarrow x_1^{q_1-p}$ in all matrices $\varphi_k$, and    
similarly the minimal free resolution of $R/I_{\rm sft}$ is
obtained from the resolution of $M=R[x_0]/I_{\rm pol}$ by making
the substitution $x_0\rightarrow x_1$ in all matrices $\varphi_k$. 
The general case follows using Proposition~\ref{sep22-25}.
\end{remark}

We come to one of our main results.

\begin{theorem}\label{sgn-invariants} Let $I$ be a monomial ideal of
$R$ of height at least
$2$ and let ${\rm sgn}(I)$ be its signature. The following hold. 
\begin{enumerate} 
\item[\rm(a)] ${\rm depth}(R/I)={\rm depth}(R/{\rm sgn}(I))$.
\item[\rm(b)] ${\rm dim}(R/I)={\rm dim}(R/{\rm sgn}(I))$.
\item[\rm(c)] $R/I$ is a Cohen--Macaulay ring $($resp. Gorenstein ring$)$ if and only 
if $R/{\rm sgn}(I)$ is a Cohen--Macaulay ring $($resp. Gorenstein ring$)$.
\item[\rm(d)] ${\rm reg}(R/I)\geq{\rm reg}(R/{\rm sgn}(I))$.
\end{enumerate}
\end{theorem}

\begin{proof} (a)--(c) follow from Proposition~\ref{carlitos-vila} and recursively
applying Lemma~\ref{depth-dim-inv}.

(d) Let $I_{\rm pol}$ and $I_{\rm sft}$ be the ideals of
Eqs.~\eqref{pol-def} and \eqref{sft-def}, respectively. By 
Proposition~\ref{carlitos-vila}, it suffices to show that 
${\rm reg}(R/I)\geq{\rm reg}(R/I_{\rm sft})$. Let $\varphi_k$ be any
of the matrices that appear in the minimal free resolution 
of $M=R[x_0]/I_{\rm pol}$ given in Eq.~\eqref{dec25-25}, let $u$ be
any of the column vectors of $\varphi_k$ such that $x_0$ occurs in
$u$, and let $v$ be the vector obtained from $u$ by making
$x_0=x_1^{q_1-p}$. Since $u$ is homogeneous and the entries of
$\varphi_k$ are either $0$ or a monomial, we may assume that 
the $j$-th entry of $u$ is $x_0^\ell x^a$, where $x_0\notin{\rm
supp}(x^a)$ and $\ell\geq 1$. Then, 
\begin{align*}
\deg(u)&=\deg(x_0^\ell x^a)+\deg(e_j),\\
\deg(v)&=\deg(x_1^{\ell(q_1-p)}x^a)+\deg(e_j)=
\deg((x^ax_1^{\ell})x_1^{\ell(q_1-p-1)})+\deg(e_j)\\
&=\deg(u)+\ell(q_1-p-1)\geq \deg(u)+1, 
\end{align*}
where $e_j$ is the $j$-th unit vector. The degrees of the columns of $\varphi_k$ that
do not contain $x_0$ do not change under the substitution
$x_0\rightarrow x_1^{q_1-p}$, Hence, as the minimal free resolutions 
of $R/I$  and $R/I_{\rm sft}$ are 
obtained from $M=R[x_0]/I_{\rm pol}$ by making
the substitution $x_0\rightarrow x_1^{q_1-p}$ and $x_0\rightarrow x_1$
in all matrices
$\varphi_k$, respectively, we obtain that ${\rm reg}(R/I)\geq{\rm
reg}(R/I_{\rm sft})$. 
\end{proof}

\begin{corollary}\cite[Proposition~5]{cm-oriented-trees}\label{coro1}
Let $I$ be a squarefree monomial ideal of $R$ and
let $\{d_i\}_{i=1}^n$ be a sequence of positive integers. If $J$ is the
ideal of $R$ generated by all monomials $x_1^{d_1}\cdots x_n^{d_n}$
such that ${\rm supp}(x_1^{d_1}\cdots x_n^{d_n})\in G(I)$, then 
${\rm sgn}(J)=I$ and $R/J$ is Cohen--Macaulay $($resp. Gorenstein$)$ if
and only if $R/I$ is Cohen--Macaulay $($resp. Gorenstein$)$. 
\end{corollary}

\begin{proof} If the height of $I$ is $1$, then $I=(x_i)$ for some $I$
and the result is clear. Thus, we may assume that the height of $I$ is
at least $2$. By Theorem~\ref{sgn-invariants} it suffices to show 
that ${\rm sgn}(J)=I$. Let $F_1,\ldots,F_n$ be the rows of the
incidence matrix of $I$, then $d_1F_1,\ldots,d_nF_n$ are the rows of
the incidence matrix of $J$. As $I$ is squarefree, for each $F_i$, the non-zero
entries of $d_iF_i$ are equal to $d_i$. Hence, ${\rm sgn}(d_iF_i)=F_i$
and consequently ${\rm sgn}(J)=I$.
\end{proof}

\begin{definition}\rm
A \textit{weighted oriented graph} $D$ is a simple
graph $G$ in which each edge $\{u,v\}$ of $G$ has a direction $(u,v)$ or
$(v,u)$ and each vertex $v$ of $G$ has a positive weight
$w(v)\in\mathbb{N}_+$. The set $E(D)$ of
directions is the \textit{edge set} of $D$ and 
the \textit{vertex set} of $D$ is the vertex set of $G$.
\end{definition} 

\begin{definition}\rm\cite{cm-oriented-trees,WOG}  
Let $D$ be a weighted oriented graph with vertex set
$V(D)=\{x_1,\ldots,x_n\}$ and let $w_i:=w(x_i)$.   
The
 \textit{edge ideal} of $D$ is the ideal of $R$ given by 
$$I(D):=(\{x_{i}x_{j}^{w_j}\mid (x_{i},x_{j})\in E(D)\}).$$
\end{definition}

If $w_i=1$ for each $x_i$, then $I(D)$ is the usual edge
ideal $I(G)$ of the graph $G$. The interest in studying $I(D)$ comes 
from the study of Reed-Muller typed codes
\cite{min-dis-generalized,oriented-graphs}, because
one can use $I(D)$ to 
compute and estimate the basic parameters of some of these codes 
\cite{carvalho-lopez-lopez,hilbert-min-dis}. 

\begin{corollary}\cite[Corollary~6]{cm-oriented-trees}\label{oct20-17} 
Let $I=I(D)$ be the edge ideal of a
weighted oriented graph with vertices $x_1,\ldots,x_n$ 
and let $w_i$ be the weight of $x_i$. If $\mathcal{U}$ is the
weighted oriented graph  obtained from $D$ by assigning weight $2$ to every $x_i$
with $w_i\geq 2$, then ${\rm sgn}(I)=I(\mathcal{U})$ and 
$I$ is Cohen--Macaulay $($resp. Gorenstein$)$ if and only 
if $I(\mathcal{U})$ is Cohen--Macaulay $($resp. Gorenstein$)$. 
\end{corollary}

\begin{proof} By Theorem~\ref{sgn-invariants} it suffices to show 
that ${\rm sgn}(I)=I(\mathcal{U})$. Let $A$ be the incidence matrix 
of $I$ and let $F_i$ be the $i$-th row of $A$. Note that each entry of
$F_i$ is equal to $0,1$ or $w_i$. Thus, ${\rm sgn}(F_i)$ is obtained
from $F_i$ by replacing $w_i$ by $2$ if $w_i\geq 2$, and consequently
${\rm sgn}(I)=I(\mathcal{U})$.
\end{proof}

\begin{remark} Let $x_j$ be a sink of $D$ (i.e., a vertex with only
incoming edges). If $x_ix_j^{w_j}\in I(D)$, then the exponent $w_j$
reduces to $1$ in ${\rm sgn}(I(D))$ if ${\rm ht}(I(D))\geq 2$ and
reduces to $0$ if all edges of $D$ contain $x_j$, see
Example~\ref{reviewer}.
\end{remark}

\begin{proposition}\label{reviewer1}
If $I$ is a squarefree monomial ideal of $R$ of height at
least $2$, then ${\rm sgn}(I)=I$. 
\end{proposition}

\begin{proof} Let $A$ be the incidence matrix of $I$ and let 
$c=\{c_1,\ldots,c_q\}$ be the multiset of entries of any non-zero 
row of $A$. Since $I$ has height at least $2$, $c$ has at least one
element equal to $0$ and at least one element equal to $1$. Then, 
the list of entries of $c$ listed in ascending order is
$m=\{m_0,m_1\}$, where $m_0=0$ and $m_1=1$. Let $j_i$ be the position
of $c_i$ relative to $m$, that is, $c_i=m_{j_i}$. If $c_i=1=m_1=m_{j_i}$,
one has $j_i=1$, and if $c_i=0=m_0=m_{j_i}$, one has $j_i=0$. Thus,
$j_i=1$ if $c_i=1$ and $j_i=0$ if $c_i=0$. Thus, 
${\rm sgn}(c)=\{j_1,\ldots,j_q\}=c$ and consequently we obtain the
equality ${\rm sgn}(I)=I$.
\end{proof}

\begin{lemma}\label{nov10-25}
If $I\subset R$ is a monomial ideal and $x_i$ is a variable in $R$, 
then ${\rm sgn}(I)={\rm sgn}(x_iI)$.
\end{lemma}

\begin{proof} We may assume $i=1$. Let $a=(a_{1,1},\ldots,a_{1,q})$ be
the first row of the incidence matrix $A$ of $I$. 
Then $\overline{a}=(a_{1,1}+1,\ldots,a_{1,q}+1)$ is the first row of
the incidence matrix of $x_1I$. If we list the distinct entries of $a$ in ascending order 
$$m=\{m_0,m_1,\ldots,m_r\},\ m_0<m_1<\cdots<m_r,$$ 
then, the list of entries of $\overline{a}$ in ascending order
is $\overline{m}=\{m_0+1,m_1+1,\ldots,m_r+1\}$. Hence, 
${\rm sgn}(a)={\rm sgn}(\overline{a})$ because if $j_i$ is the
position 
of $a_{1,i}$ relative to $m$, that is, $a_{1,i}=m_{j_i}$, $0\leq
j_i\leq r$, then $j_i$ is also the position of $a_{1,i}+1$ relative to 
$\overline{m}$. Thus, ${\rm sgn}(a)={\rm sgn}(\overline{a})$, and 
consequently ${\rm sgn}(I)={\rm sgn}(x_1I)$ since the other rows of the
incidence matrices of $I$ and $x_1I$ are equal.
\end{proof}

\begin{lemma}\label{nov30-25} Given integers $n\geq 1$ and $q\geq 1$,
the set ${\rm Sgn}({n,q})$ of 
all ${\rm sgn}(I)$ such that $I$ is minimally generated by $q$
monomials in $n$ variables is finite.
\end{lemma}

\begin{proof} Let ${\rm sgn}(I)\in{\rm Sgn}(n,q)$ and let
$B=(w_{i,j})$ be the incidence matrix 
of ${\rm sgn}(I)$. Then, $w_{i,j}\leq q-1$ for all $i,j$, and consequently 
$|{\rm Sgn}({n,q})|\leq q^{nq}$.
\end{proof}

\begin{proposition}\label{sgn-ht}
If $I$ is a non-principal monomial ideal of $R$, then ${\rm ht}({\rm
sgn}(I))\geq 2$.
\end{proposition}

\begin{proof} Let $A$ be the incidence matrix of $I$ and let $F_i$ be
the $i$-th row of $A$. If ${\rm sgn}(F_i)=0$ for $i=1,\ldots,n$, then
for each $i$ there is $k_i\in\mathbb{N}$ such that $F_i=k_ie_i$.
Hence, all columns of $A$ are equal to $(k_1,\ldots,k_n)^\top$, where
$n$ is the number of variables of $R$. Consequently, $A$ has only one
column and $I$ is a principal ideal, a contradiction. Thus,  
${\rm sgn}(A)$ has at least one non-zero row. We may assume
that ${\rm sgn}(F_1),\ldots,{\rm sgn}(F_k)$ are the non-zero rows of
${\rm sgn}(A)$. Then, for each $1\leq i\leq k$, ${\rm sgn}(F_i)$ has
at least one entry equal to $0$. Thus, the minimal generators of 
${\rm sgn}(I)$ cannot have a common variable, or equivalently, all 
associated primes of ${\rm sgn}(I)$ have height at least $2$.
Therefore, the height of ${\rm sgn}(I)$ is at least $2$.    
\end{proof}

We now deal with the depth for non-principal monomial ideals of height
$1$.

\begin{proposition}\label{depth-height1} Let $I\subset R$ be a monomial ideal of height $1$. If
$I$ is not a principal ideal of $R$ and $f=\gcd(G(I))$ is the greatest
common divisor of $G(I)$, then there is
a monomial ideal $L$ of height $\geq 2$ such that 
$I=fL$, ${\rm
depth}(R/I)={\rm depth}(R/L)$, and ${\rm sgn}(I)={\rm sgn}(L)$.
\end{proposition}

\begin{proof} The existence of $L$ is clear by the choice of $f$. Consider the short exact sequence 
$$
0\longrightarrow R/(I\colon f)\stackrel{f}{\longrightarrow} R/I
\longrightarrow R/(I,f)\longrightarrow 0,
$$
By \cite[Theorem~3.1]{Caviglia-et-al} (cf. \cite[Corollary
2.12 (vi)]{depth-monomial}), 
one has that either
$${\rm depth}(R/I)={\rm
depth}(R/(I\colon f))\ \mbox{ or }\ {\rm depth}(R/I)={\rm
depth}(R/(I,f)).
$$ 
\quad Note that $(I\colon f)=L$, $(I,f)=(f)$, and ${\rm ht}(L)\geq 2$
since $f=\gcd(G(I))$. Then, one has
$$
{\rm depth}(R/I)={\rm
depth}(R/L)\ \mbox{ or }\ {\rm depth}(R/I)={\rm
depth}(R/(f))=n-1, 
$$
where $n=\dim(R)$. We argue by contradiction assuming that ${\rm
depth}(R/I)\neq {\rm depth}(R/L)$. Then, ${\rm depth}(R/I)=n-1$ and,
by \cite[Corollary 2.12(ii)]{depth-monomial}, we get 
$$n-1={\rm depth}(R/I)\leq {\rm
depth}(R/(I\colon f))={\rm depth}(R/L).
$$ 
Thus, $n-1={\rm depth}(R/I)<{\rm depth}(R/L)$, and consequently ${\rm
depth}(R/L)=n$. Hence, $\dim(R/L)$ is equal to $n$ because ${\rm
depth}(R/L)\leq\dim(R/L)\leq n$. Therefore, ${\rm ht}(L)=0$ and
$L=(0)$, a contradiction. This proves that ${\rm
depth}(R/I)={\rm depth}(R/L)$. From Lemma~\ref{nov10-25}, we obtain the
equality ${\rm
sgn}(I)={\rm sgn}(L)$. 
\end{proof}

\begin{remark}\label{reviewer2}
With the notation of Proposition~\ref{depth-height1}, 
the equality $fL=(f)\cap L$ is not true. For instance letting 
$f=x_1x_2$ and $L=(x_1,x_2)$, one has 
$$
fL=(x_1^2x_2,\, x_1x_2^2)\subsetneq(x_1x_2)=(f)\cap L.
$$
\quad Below we classify when the
equality occurs.
\end{remark}
\begin{lemma}\label{fL}
Let $L$ be an ideal of $R$ and $f\in R$. Then, $fL\subset(f)\cap L$
with equality if and only if $(L\colon f)=L$. 
\end{lemma}

\begin{proof} The inclusions $fL\subset(f)\cap L$ and $L\subset(L\colon
f)$ are clear. Thus, it suffices to show that $fL\supset(f)\cap L$ if and only if 
$(L\colon f)\subset L$.

$\Rightarrow$) Take $g\in(L\colon f)$. Then, $gf\in L\cap(f)$, and
consequently $gf\in fL$ and $g\in L$.

$\Leftarrow$) Take $g\in(f)\cap L$. Then, $g=hf$ and $g\in L$, and
consequently $h\in(L\colon f)\subset L$. Thus, one has $g\in fL$ and 
the proof is complete.
\end{proof}

The next result relates the regularity of the ideals $I$ and $L$ of 
Proposition~\ref{depth-height1} for non-principal monomial ideals of height
$1$.

\begin{proposition}\label{reg-height1} Let $I\subset R$ be a monomial ideal of height $1$. If
$I$ is not a principal ideal of $R$ and $f=\gcd(G(I))$ is the greatest
common divisor of $G(I)$, then there is
a monomial ideal $L$ of height $\geq 2$ such that 
$I=fL$, ${\rm reg}(R/I)={\rm reg}(R/L)+\deg(f)$, and ${\rm sgn}(I)={\rm sgn}(L)$.
\end{proposition}

\begin{proof} The existence of $L$ is clear by the choice of $f$. We
may assume that $f=x_{i_1}x_{i_2}\cdots x_{i_k}$, $i_1\leq\cdots\leq
i_k$, and write $I=x_{i_1}L_1$, where $L_1=x_{i_2}\cdots x_
{i_k}L$ if $i_k>1$ and $L_1=L$ if $i_k=1$. According to 
\cite{Caviglia-et-al,Dao-Huneke-Schweig}, one has that either 
$$
{\rm reg}(R/I)={\rm reg}(R/(I\colon x_{i_1}))+1\ \mbox{ or }\ 
{\rm reg}(R/I)={\rm reg}(R/(x_{i_1},I)).
$$ 
\quad Note that $(I\colon x_{i_1})=L_1$ and $(I,x_{i_1})=(x_{i_1})$.
Thus,  either ${\rm reg}(R/I)={\rm reg}(R/L_1)+1$ or ${\rm
reg}(R/I)={\rm reg}(R/(x_{i_1}))$. The regularity  of $R/(x_{i_1})$ is
equal to $0$ because $R/(x_{i_1})$ is a polynomial ring over the field
$K$. On the other hand, since $I$ has height $1$ and is not a
principal ideal, it has a minimal generator of degree at least $2$.
Thus, ${\rm reg}(R/I)\geq 1$. Therefore, one has the equality ${\rm reg}(R/I)={\rm
reg}(R/L_1)+1$ and the statement about the regularity follows by induction on $k$.  
The equality ${\rm sgn}(I)={\rm sgn}(L)$ follows from
Proposition~\ref{depth-height1}.
\end{proof}

\begin{lemma}\label{sgn2=sgn}
If $I$ is a monomial ideal of $R$, then ${\rm sgn}({\rm sgn}(I))={\rm sgn}(I)$,
that is, $I$ is in the signature class $[{\rm sgn}(I)]$ of ${\rm
sgn}(I)$ and ${\rm sgn}(L)={\rm sgn}(I)$ for all $L\in[{\rm sgn}(I)]$.   
\end{lemma}

\begin{proof} Case (I) Assume that ${\rm ht}(I)\geq 2$. Let $F_i$ be
the $i$-th row of the incidence matrix $A$ of $I$. Since $F_i$ has at
least one entry equal to $0$, the list of distinct entries of ${\rm
sgn}(F_i)$ is $\{0,1,2,\ldots,\ell_i\}$ for some $\ell_i$, and
consequently ${\rm sgn}({\rm sgn}(F_i))={\rm sgn}(F_i)$. Thus, 
${\rm sgn}({\rm sgn}(I))={\rm sgn}(I)$. 

Case (II) Assume that $I$ is not principal ideal and ${\rm ht}(I)=1$.
Letting $f=\gcd(G(I))$, by Proposition~\ref{depth-height1}, there is a
monomial ideal $L$ of height at least $2$ such that 
$I=fL$, ${\rm
depth}(R/I)$ is equal to ${\rm depth}(R/L)$, and ${\rm sgn}(I)={\rm
sgn}(L)$. Hence, applying Case (I) to $L$, one has 
$$
{\rm sgn}({\rm sgn}(I))={\rm sgn}({\rm sgn}(L))={\rm sgn}(L)=L.
$$

Case (III) Assume that $I$ is a principal ideal. Then, ${\rm
sgn}(I)=(x^{0})=R$ and ${\rm sgn}(R)=R$. Thus one has,
$$
{\rm sgn}({\rm sgn}(I))={\rm sgn}(R)=R={\rm sgn}(I),
$$
and the proof is complete.
\end{proof}

We come to our main result on the depth of quotient rings of monomial
ideals.

\begin{theorem}\label{depth=sgn}
Let $I$ be a non-principal monomial ideal of $R$. Then, 
$$
{\rm depth}(R/I)={\rm depth}(R/{\rm sgn}(I)).
$$
\end{theorem}

\begin{proof} By Theorem~\ref{sgn-invariants}(a), we may assume that the height of
$I$ is $1$. Then, by Proposition~\ref{depth-height1} and letting
$f=\gcd(G(I))$, 
there is
a monomial ideal $L$ of height $\geq 2$ such that 
$I=fL$, ${\rm
depth}(R/I)={\rm depth}(R/L)$, and ${\rm sgn}(I)={\rm sgn}(L)$.
Hence, applying Theorem~\ref{sgn-invariants}(a) to the ideal $L$, we
get the equalities:
$$
{\rm depth}(R/I)={\rm depth}(R/L)={\rm depth}(R/{\rm sgn}(L))={\rm depth}(R/{\rm sgn}(I)).
$$
\quad Thus, ${\rm depth}(R/I)={\rm depth}(R/{\rm sgn}(I))$ and the
proof is complete. 
\end{proof}

The following result shows that the signature lowers the regularity
of $R/I$.  

\begin{theorem}\label{reg-sgn}
Let $I$ be a non-principal monomial ideal of $R$. Then, 
$$
{\rm reg}(R/I)\geq{\rm reg}(R/{\rm sgn}(I)).
$$
\end{theorem}

\begin{proof} By Theorem~\ref{sgn-invariants}(d), we may assume that the height of
the ideal $I$ is equal to $1$. Then, by Proposition~\ref{reg-height1} and letting
$f=\gcd(G(I))$, there is
a monomial ideal $L$ of height $\geq 2$ such that 
$I=fL$, ${\rm reg}(R/I)={\rm reg}(R/L)+\deg(f)$, and ${\rm sgn}(I)={\rm sgn}(L)$.
Hence, applying Theorem~\ref{sgn-invariants}(d) to the ideal $L$, we
get the inequalities:
$$
{\rm reg}(R/I)={\rm reg}(R/L)+\deg(f)\geq {\rm reg}(R/{\rm
sgn}(L))+\deg(f)\geq {\rm reg}(R/{\rm sgn}(I)).
$$
\quad Thus, ${\rm reg}(R/I)\geq{\rm reg}(R/{\rm sgn}(I))$ and the
proof is complete. 
\end{proof}

\subsection{The irreducible decomposition of the signature}
Let $I$ be a monomial ideal of height at least $2$. In this part, 
we show how to obtain the irreducible decomposition of ${\rm sgn}(I)$
from the irreducible decomposition of $I$ by recursively 
applying the shift operation $I\mapsto I_{\rm sft}$.

Let $I$ and $I_{\rm sft}$ be the ideals of
Eqs.~\eqref{i-def} and \eqref{sft-def}, respectively, and let 
$I=\bigcap_{j=1}^r\mathfrak{q}_j$ be the irreducible 
decomposition of $I$ \cite[Theorem~6.1.17]{monalg-rev}, where each
$\mathfrak{q}_j$ is generated by powers of some of the $x_i$'s variables. For
each $\mathfrak{q}_j$, we let 
$$
\mathfrak{q}_j':=
\begin{cases}
((G(\mathfrak{q}_j)\setminus\{x_1^{q_i}\})
\textstyle\bigcup\{x_1^{q_i-q_1+p+1}\})&\mbox{if }x_1^{q_i}\in
G(\mathfrak{q}_j)\mbox{ for some }1\leq i\leq s,\\
\mathfrak{q}_j&\mbox{if }x_1^{q_i}\notin
G(\mathfrak{q}_j)\mbox{ for all }1\leq i\leq s,
\end{cases}
$$
that is, each ideal $\mathfrak{q}_j'$ is obtained from
$G(\mathfrak{q}_j)$ by 
making the substitution $x_1^{q_i}\mapsto x_1^{q_i-q_1+p+1}$. By
construction, $x_1^{q_i}\in G(\mathfrak{q}_j)$ if and only
if $x_1^{q_i-q+p+1}\in G(\mathfrak{q}_j')$. 
Note that $\mathfrak{q}_j\subset\mathfrak{q}_j'$ for all $j$
because $q_i\geq q_i-q_1+p+1\geq 1$ for all $i$ and 
${\rm rad}(\mathfrak{q}_j)={\rm rad}(\mathfrak{q}_j')$ for all $j$, and consequently one
has the inclusion $I\subset \bigcap_{j=1}^r\mathfrak{q}_j'$. 

The following lemma is a sort of duality theorem showing that 
the powers of variables that occur in $G(I)$ is equal to 
$G(\mathfrak{q}_1)\cup\cdots\cup G(\mathfrak{q}_r)$.

\begin{lemma}\cite[Lemma~1]{cm-oriented-trees}\label{duality-of-exponents} 
Let $I$ be a monomial ideal of $R$, with 
$G(I)=\{x^{v_1},\ldots,x^{v_q}\}$, let 
$v_i=(v_{i,1},\ldots,v_{i,n})$ for $i=1,\ldots,q$, and let
$I=\textstyle\bigcap_{j=1}^r\mathfrak{q}_j$ 
be its irreducible decomposition. Then
$$
\{x_j^{v_{i,j}}\vert\, v_{i,j}\geq
1\}=G(\mathfrak{q}_1)\textstyle\bigcup\cdots\bigcup
G(\mathfrak{q}_r).
$$ 
\end{lemma}

We say the irreducible decomposition
$I=\bigcap_{j=1}^r\mathfrak{q}_j$ of a monomial ideal $I$ is
\textit{tight} if for each variable $x_k$, the set of powers of $x_k$
that occur in $G(\mathfrak{q}_1)\textstyle\bigcup\cdots\bigcup
G(\mathfrak{q}_r)$ listed in ascending degree order is
$\{x_k,x_k^2,\ldots,x_k^{\ell_i}\}$ for some $\ell_i$. Note that by
Lemma~\ref{duality-of-exponents}, the incidence matrix $A$ of $I$ is
tight if and only if the irreducible decomposition of $I$ is tight. 

\begin{proposition}\label{irred-i-sgni}
If $I=\bigcap_{j=1}^r\mathfrak{q}_j$ is the irreducible 
decomposition of $I$, then $I_{\rm sft}=\bigcap_{j=1}^r\mathfrak{q}_j'$
is the irreducible decomposition
of $I_{\rm sft}$.
\end{proposition}

\begin{proof} To show the inclusion $I_{\rm
sft}\subset\bigcap_{j=1}^r\mathfrak{q}_j'$, take $x^a\in I_{\rm sft}$.
Let $\mathfrak{q}_j$ be any irreducible component of $I$ and let 
$\mathfrak{q}_j'$ be its corresponding ideal. According to
Eq.~\eqref{sft-def}, we have two
cases to consider. 

Case (A) $x^a=x_1^{j_1}x^{\gamma_{i_1,j_1}}$. Then, 
$x^a\in I$. Hence, $x^a\in\mathfrak{q}_j'$ because $I\subset
\bigcap_{j=1}^r\mathfrak{q}_j'$. 

Case (B) $x^a=x_1^{q_i-q_1+p+1}x^{\epsilon_i}$. Then,  
$x_1^{q_1-p-1}x^a=x_1^{q_i}x^{\epsilon_i}\in I\subset\mathfrak{q}_j$,
and there is $x_k^\ell\in G(\mathfrak{q}_j)$ such that 
$x_1^{q_i}x^{\epsilon_i}=x_k^\ell x^\theta$. If $k\neq 1$, then
$x^{\epsilon_i}=x_k^\ell
x^{\theta_1}\in\mathfrak{q}_j\subset\mathfrak{q}_j'$ because
$x_1\notin{\rm supp}(x^{\epsilon_i})$, and consequently
$x^a\in\mathfrak{q}_j'$. Thus, we may assume that $k=1$. Then, 
$x_1^{q_i}x^{\epsilon_i}=x_1^\ell x^\theta$, and consequently 
$q_i\geq \ell$. If $q_i=\ell$, since $x_1^{q_i}=x_1^\ell\in
G(\mathfrak{q}_j)$, then, by construction of $\mathfrak{q}_j'$, $x_1^{q_i-q_1+p+1}\in G(\mathfrak{q}_j')$ 
and $x^a\in\mathfrak{q}_j'$. If $q_i>\ell$ and $\ell=q_k$ for some
$k$, then, by construction of $\mathfrak{q}_j'$, one has $x_1^{q_k-q_1+p+1}\in\mathfrak{q}_j'$, and since
$q_i>q_k$, we get $x_1^{q_i-q_1+p+1}\in\mathfrak{q}_j'$ and
$x^a\in\mathfrak{q}_j'$. If $q_i>\ell$
and $\ell\neq q_k$ for all $k$, then, by
Lemma~\ref{duality-of-exponents}, 
$x_1^\ell$ has to appear in a minimal generator of $I$ and, 
by Eq.~\eqref{i-def}, we obtain $1\leq \ell\leq p$. Then, 
$q_i-q_1+p+1\geq \ell$, and consequently
$x^a\in\mathfrak{q}_j\subset\mathfrak{q}_j'$. 

To show the inclusion $I_{\rm
sft}\supset\bigcap_{j=1}^r\mathfrak{q}_j'$, take
$x^a\in\bigcap_{j=1}^r\mathfrak{q}_j'$. We claim that
$x_1^{q_1-p-1}x^a\in I$. Take any $\mathfrak{q}_j$, it suffices to
show that $x_1^{q_1-p-1}x^a\in\mathfrak{q}_j$. We can write
$x^a=x_k^\ell x^\theta$ for some $x_k^\ell\in G(\mathfrak{q}_j')$. If
$k\neq 1$, by construction of $\mathfrak{q}_j'$, $x_k^\ell\in
G(\mathfrak{q}_j)$ and $x_1^{q_1-p-1}x^a\in\mathfrak{q}_j$. If $k=1$
and $\ell\neq q_i-q_1+p+1$ for all $i$, then
$\mathfrak{q}_j=\mathfrak{q}_j'$ and
$x_1^{q_1-p-1}x^a\in\mathfrak{q}_j$. If $k=1$
and $\ell=q_i-q_1+p+1$ for some $i$, then
$x_1^\ell\in\mathfrak{q}_j'$ and
$x_1^{q_1-p-1}x_1^\ell=x_1^{q_i}\in\mathfrak{q}_j$. Thus,
$x_1^{q_1-p-1}x^a\in\mathfrak{q}_j$. Therefore, $x_1^{q_1-p-1}x^a\in
I$. By Eq.~\eqref{i-def}, we have two cases to consider.

Case (A) $x_1^{q_1-p-1}x^a=x_1^{j_1}x^{\gamma_{i_1,j_1}}x^\theta$,
$0\leq j_1\leq p$, $x_1\notin{\rm supp}(x^{\gamma_{i_1,j_1}})$. Cancelling out $x_1$ from both sides of the equality
and letting $a=(a_1,\ldots,a_n)$, 
we get $x_2^{a_2}\cdots x_n^{a_n}=x^{\gamma_{i_1,j_1}}x^{\theta_2}$. 
If $a_1\geq p$, then $x^a\in I\subset I_{\rm sft}$ because
$x_1^{a_1}x^{\gamma_{i_1,j_1}}\in I$. If $a_1<p$, then
$q_i-q_1+p+1>p>a_1$. Take any $\mathfrak{q}_j'$. Since $x^a\in
\mathfrak{q}_j'$, we can write $x^a=x_k^\ell x^{\theta_3}$ for some
$x_k^\ell\in G(\mathfrak{q}_j')$. If $k\neq 1$, then
$x_k^\ell\in\mathfrak{q}_j$
 and $x^a\in\mathfrak{q}_j$. If $k=1$, then $a_1\geq \ell$,
 $x_1^\ell\in G(\mathfrak{q}_j')$ and since $q_i-q_1+p+1>a_i\geq \ell$
 for all $i$, $x_1^\ell\in\mathfrak{q}_j$ and $x^a\in\mathfrak{q}_j$.
 Therefore, we have shown that
 $x^a\in\bigcap_{j=1}^r\mathfrak{q}_j=I$ and since $I\subset I_{\rm
 sft}$, we get $x^a\in{\rm I}_{\rm sft}$.

Case (B) $x_1^{q_1-p-1}x^a=x_1^{q_i}x^{\epsilon_i}x^{\theta_1}$,
$1\leq i\leq s$. Then,
$x^a=x_1^{q_i-q_1+p+1}x^{\epsilon_i}x^{\theta_1}\in I_{\rm sft}$.
 \end{proof}

The following theorem complements the main result of
\cite[Theorem~3.11]{signature} showing that two monomial ideals with 
the same signature have essentially the same irreducible
decomposition.  

\begin{theorem}\label{irred-i-sgni-theo} Let $I$ be a monomial ideal of $R$ of height $\geq 2$. 
If $I=\textstyle\bigcap_{j=1}^r\mathfrak{q}_j$ is the 
irreducible decomposition of $I$, then for each $\mathfrak{q}_j$ there
is an irreducible component $\mathfrak{q}_j'$ of ${\rm
sgn}(I)$ such that $\mathfrak{q}_j$ and $\mathfrak{q}_j'$ have the same
radical and $\mathfrak{q}_j\subset\mathfrak{q}_j'$. Furthermore, 
${\rm sgn}(I)=\textstyle\bigcap_{j=1}^r\mathfrak{q}_j'$ is the
irreducible decomposition of ${\rm sgn}(I)$ and can be obtained by 
recursively applying the operation $I\mapsto I_{\rm sft}$ and making 
substitutions of the form $x_k^{q_i}\mapsto x_k^{q_i-q_1+p+1}$ 
in the irreducible decomposition of $I$.
\end{theorem}
 
\begin{proof} This follows from Proposition~\ref{irred-i-sgni} and its proof.
\end{proof}

\section{Associated primes and v-number of a monomial ideal and its
signature}

Let $I$ be a monomial ideal of $R$ of height at least $2$, let $A$
be its incidence matrix, and let $G(I)$ be the minimal generating set
of $I$. In this section we show that the associated primes of $I$ and its signature
${\rm sgn}(I)$ are the same, and show that the v-number of
${\rm sgn}(I)$ is at most the v-number of $I$. 
To avoid repetitions we continue to use the notation and assumptions introduced in
Section~\ref{intro-section} and
Subsection~\ref{tight-matrix-subsection}. 

For convenience recall that we are assuming that $A$ is
not tight, and we may assume that the
first row of $A$ is not tight. Then, after listing the monomials of
$G(I)$ in ascending order with respect to
the powers of $x_1$, there are integers $p\geq 0$, $q_1\leq\cdots\leq q_s$, and
$k_0,\ldots,k_p$, such that 
$q_1-p\geq 2$, $k_i\geq 1$ for $i=0,\ldots,p$, and we can write the
minimal generating set of $I$ as 
\begin{align}\label{dec2-25-1}
G(I)=&\big[\textstyle\bigcup_{j=0}^p\big\{x_1^jx^{\gamma_{1,j}},\ldots,x_1^jx^{\gamma_{i,j}},
\ldots,x_1^jx^{\gamma_{k_j,j}}\big\}
\big]
\bigcup\big\{x_1^{q_i}x^{\epsilon_i}\big\}_{i=1}^s,
\end{align}
where $x_1\notin{\rm supp}(x^{\gamma_{i,j}})$ for all $i,j$ and
$x_1\notin{\rm supp}(x^{\epsilon_i})$ for all $i$. The minimal
generating set of the shift ideal $I_{\rm sft}$ is given by 
\begin{align}\label{dec2-25-2}
G(I_{\rm sft})=\big[\textstyle\bigcup_{j=0}^p\big\{x_1^jx^{\gamma_{1,j}},
\ldots,x_1^jx^{\gamma_{i,j}},\ldots, x_1^jx^{\gamma_{k_j,j}}\big\}\big]
\bigcup\big\{x_1^{\ell_i}x^{\epsilon_i}\big\}_{i=1}^s,
\end{align}
where $\ell_i=q_i-q_1+p+1$ for $i=1,\ldots,s$. We can successively
apply the shift operator to each variable $x_i$ 
to obtain the signature ${\rm sgn}(I)$ of $I$.

We come to another of our main results.
\begin{theorem}\label{ass-v-number}
Le $I$ be a monomial ideal of height $\geq 2$ and let ${\rm v}(I)$ be
its ${\rm v}$-number. Then
\begin{enumerate}
\item[\rm(i)] ${\rm Ass}(I)={\rm Ass}({\rm sgn}(I))$.
\item[\rm(ii)] ${\rm v}({\rm sgn}(I))\leq {\rm v}(I)$.
\item[\rm(iii)] $I$ is unmixed if and only if ${\rm sgn}(I)$ is
unmixed. 
\end{enumerate}
\end{theorem}

\begin{proof} (i) By Proposition~\ref{carlitos-vila} it suffices to
show that ${\rm Ass}(I)={\rm Ass}(I_{\rm sft})$. We let
$J:=I_{\rm sft}$ and recall that $I\subset J$. 

(A) To show the inclusion ${\rm Ass}(I)\subset{\rm Ass}(J)$, take
$\mathfrak{p}\in {\rm Ass}(I)$, that is, $(I\colon x^a)=\mathfrak{p}$
for some $x^a$, where $a=(a_1,\ldots,a_n)$. 

$(\mathrm{A}_1)$ $x^a\notin J$. It suffices to show that $(J\colon
x^a)=\mathfrak{p}$. We have two subcases to consider.

$(\mathrm{A}_{1.1})$ $x_1\notin\mathfrak{p}$. To show the inclusion
$\mathfrak{p}\subset(J\colon x^a)$, take $x_k\in\mathfrak{p}$. Then, 
$x_kx^a\subset I\subset J$ and $x_k\in(J\colon x^a)$. To show the
inclusion $\mathfrak{p}\supset(J\colon x^a)$, take $x^c\in(J\colon
x^a)$. Then, $x^cx^a\in J$. If $x^cx^a\in I$, then $x^c\in(I\colon
x^a)=\mathfrak{p}$ and $x^c\in\mathfrak{p}$. If $x^cx^a\notin I$, then
from Eq.~\eqref{dec2-25-2}, 
we can write $x^cx^a=x_1^{\ell_j}x^{\epsilon_j}x^\theta$, where
$\ell_j=q_j-q_1+p+1$ for some $1\leq j\leq s$. Hence,
$$
x^cx_1^{q_1-p-1}x^a=x_1^{\ell_j+q_1-p-1}x^{\epsilon_j}x^\theta=
x_1^{q_j}x^{\epsilon_j}x^\theta,
$$
and $x^cx_1^{q_1-p-1}x^a\in I$. Thus, $x^cx_1^{q_1-p-1}\in(I\colon
x^a)=\mathfrak{p}$ and $x^c\in\mathfrak{p}$ because
$x_1\notin\mathfrak{p}$. Therefore, one has $\mathfrak{p}=(J\colon
x^a)$ and $\mathfrak{p}\in{\rm Ass}(J)$.

$(\mathrm{A}_{1.2})$ $x_1\in\mathfrak{p}$. Then, $x_1\in(J\colon x^a)$ because
$x_1x^a\in I\subset J$ and, from Eq.~\eqref{dec2-25-1}, we can write 
$$
x_1x^a=
\begin{cases}
x_1^jx^{\gamma_{i,j}}x^\theta& \mbox{
for some }0\leq j\leq p\ \mbox{ or}\\
x_1^{q_j}x^{\epsilon_j}x^\theta&\mbox{
for some }1\leq j\leq s.
\end{cases}
$$
\quad Note that $x_1\notin{\rm supp}(x^\theta)$
because $x^a\notin I$ and recall that $q_j=\ell_j+q_1-p-1$. 

$(\mathrm{A}_{1.2.1})$ Let $x_1x^a=x_1^jx^{\gamma_{i,j}}x^\theta$.
Then, $0\leq a_1=j-1\leq p-1$. To show the inclusion $(J\colon
x^a)\subset\mathfrak{p}$, take $x^c\in(J\colon x^a)$, i.e., $x^cx^a\in J$. We may assume
$x_1\notin{\rm supp}(x^c)$, because otherwise $x^c\in\mathfrak{p}$. If
$x^cx^a\in I$, then $x^c\in(I\colon x^a)=\mathfrak{p}$ and
$x^c\in\mathfrak{p}$. If $x^cx^a\notin I$, then, from
Eqs.\eqref{dec2-25-1} and \eqref{dec2-25-2}, we can write 
$x^cx^a=x_1^{\ell_i}x^{\epsilon_i}x^\delta$, and $a_1\geq\ell_i\geq
p+1$, a contradiction since $a_1\leq p-1$. This proves that $(J\colon
x^a)\subset\mathfrak{p}$. To show the inclusion $(J\colon
x^a)\supset\mathfrak{p}$, take $x_k\in\mathfrak{p}$. Then,  
$x_kx^a\in I\subset J$ and $x_k\in(J\colon x^a)$. Thus, 
$\mathfrak{p}\in{\rm Ass}(J)$.

$(\mathrm{A}_{1.2.2})$ Let $x_1x^a=x_1^{q_j}x^{\epsilon_j}x^\theta$. 
Then, $a_1=q_j-1$ and noticing that $q_j-1\geq q_j-q_1+p+1=\ell_j$,
we get 
the equalities
$$
x^a=x_1^{q_j-1}x^{\epsilon_j}x^\theta=x_1^{\ell_j}x^{\epsilon_j}x^{\theta_1}.
$$
\quad Thus, $x^a\in J$, a contradiction. This means that this case
cannot occur.

Therefore from $(\mathrm{A}_{1.2.1})$ and $(\mathrm{A}_{1.2.2})$ one has 
$\mathfrak{p}=(J\colon x^a)$ and $\mathfrak{p}\in{\rm Ass}(J)$.

$(\mathrm{A}_2)$ $x^a\in J$. As $x^{a}\notin I$ and 
$\mathfrak{p}=(I\colon x^a)$, from Eqs.\eqref{dec2-25-1} and
\eqref{dec2-25-2}, there is $i$ such that 
$$x^a=x_1^{\ell_i}x^{\epsilon_i}x^\theta\ \mbox{ and }\
a_1\geq\ell_i=q_i-q_1+p+1\geq p+1. 
$$
\quad Note that
$x_1\in\mathfrak{p}$. Indeed,
$x_1^{q_i-\ell_i}x^a=x_1^{q_i}x^{\epsilon_i}x^\theta$ and consequently
$x_1^{q_i-\ell_i}x^a\in I$, that is, $x_1^{q_i-\ell_i}\in(I\colon
x^a)=\mathfrak{p}$ and $x_1\in\mathfrak{p}$. Then, $x_1x^a\in I$. If 
$x_1x^a=x_1^{j_1}x^{\gamma_{i_1,j_1}}x^{\beta}$ for some $1\leq
j_1\leq p$, then $x_1\notin{\rm
supp}(x^\beta)$ because $x^a\notin I$, 
$x^a=x_1^{j_1-1}x^{\gamma_{i_1,j_1}}x^{\beta}$, and $p<
a_1=j_1-1\leq p-1$, a contradiction. Thus, we can write
$x_1x^a=x_1^{q_{i_1}}x^{\epsilon_{i_1}}x^{\theta_1}$ with $x_1\notin{\rm
supp}(x^{\theta_1})$. Then
\begin{equation}\label{nov7-25}
x^a=x_1^{q_{i_1}-1}x^{\epsilon_{i_1}}x^{\theta_1}\ \mbox{ and }\ 
a_1=q_{i_1}-1.
\end{equation}

Let $h$ be the largest
integer such that $x_1^hx_2^{a_2}\cdots x_n^{a_n}\notin J$.
This integer is well defined because $x^a\in J$ and $x_2^{a_2}\cdots
x_n^{a_n}\notin J$; this follows readily recalling that $x^a\notin I$ and 
from Eqs.~\eqref{dec2-25-1} and \eqref{dec2-25-2}. Letting $x^b:=x_1^hx_2^{a_2}\cdots x_n^{a_n}$, one
has $x_1x^b\in J$ and $a_1>h$. If
$x_1x^b=x_1^{j_2}x^{\gamma_{i_2,j_2}}x^{\beta_2}$
for some $1\leq j_2\leq p$, 
then $x_2^{a_2}\cdots x_n^{a_n}=x^{\gamma_{i_2,j_2}}x^{\beta_3}$ and 
$x^a=x_1^{a_1}x^{\gamma_{i_2,j_2}}x^{\beta_3}\in I$ because $a_1\geq
p$, a contradiction. Hence we can write
$x_1x^b=x_1^{\ell_j}x^{\epsilon_j}x^\delta$ for some $j$. Noticing
that $x_1\notin{\rm supp}(x^\delta)$ because $x^b\notin J$, one has 
$$
x^b=x_1^{\ell_j-1}x^{\epsilon_j}x^\delta\ \mbox{ and }\ a_1>h=\ell_j-1\geq
p.
$$

In what follows we show that $\mathfrak{p}\in{\rm Ass}(J)$ by proving
that $\mathfrak{p}=(J\colon x^b)$.  

To show the inclusion $\mathfrak{p}\subset(J\colon x^b)$, take
$x_k\in\mathfrak{p}$. Then, $x_kx^a\in I$. If $k=1$, we know that
$x_1x^b\in J$ by the choice of $h$ and $x_1\in(J\colon x^b)$. Thus, we
may assume that $k\neq 1$. We consider two cases. If
$x_kx^a=x_1^jx^{\gamma_{i,j}}x^\theta$ for some $0\leq j\leq p$, then
cancelling out $x_1$ from both sides of the equality one has 
$$x_kx_2^{a_2}\cdots x_n^{a_n}=x^{\gamma_{i,j}}x^{\theta_1}
\ \mbox{and }\ x_kx^b=x_kx_1^hx_2^{a_2}\cdots
x_n^{a_n}=x_1^hx^{\gamma_{i,j}}x^{\theta_1}.
$$
\quad Thus, $x_kx^b\in I\subset J$ because $h\geq p$, and
$x_k\in(J\colon x^b)$. If $x_kx^b=x_1^{q_t}x^{\epsilon_t}x^\delta$ for
some $t$, then $a_1>h\geq q_t\geq \ell_t\geq p+1$ and cancelling out
$x_1$ from both sides of the equality one has 
\begin{equation}\label{nov8-25}
x_kx_2^{a_2}\cdots x_n^{a_n}=x^{\epsilon_t}x^{\delta_1}
\ \mbox{ and }\ x_kx^b=x_kx_1^hx_2^{a_2}\cdots
x_n^{a_n}=x_1^hx^{\epsilon_t}x^{\delta_1}.
\end{equation}
\quad We claim that $x_1^{\ell_t}x_2^{a_2}\cdots x_n^{a_n}\notin J$.
We argue by contradiction assuming that $x_1^{\ell_t}x_2^{a_2}\cdots
x_n^{a_n}\in J$. If 
$x_1^{\ell_t}x_2^{a_2}\cdots
x_n^{a_n}=x_1^{j_3}x^{\gamma_{i_3,j_3}}x^{\beta_4}$ for some 
$0\leq j_3\leq p$, then cancelling out $x_1$ one has 
$$x_2^{a_2}\cdots
x_n^{a_n}=x^{\gamma_{i_3,j_3}}x^{\beta_5} \mbox{ and }\  x^a=x_1^{a_1}x_2^{a_2}\cdots
x_n^{a_n}=x_1^{a_1}x^{\gamma_{i_3,j_3}}x^{\beta_5}.
$$
\quad Hence, $x^a\in I$ because $a_1\geq p+1$, a contradiction. If 
$x_1^{\ell_t}x_2^{a_2}\cdots
x_n^{a_n}=x_1^{\ell_\lambda}x^{\epsilon_\lambda}x^u$ for some
$\lambda$, then $\ell_t\geq\ell_\lambda$,
$q_t\geq q_\lambda$, and cancelling out $x_1$ one has 
$$x_2^{a_2}\cdots
x_n^{a_n}=x^{\epsilon_\lambda}x^{u_1} \mbox{ and }\  x^a=x_1^{a_1}x_2^{a_2}\cdots
x_n^{a_n}=x_1^{a_1}x^{\epsilon_\lambda}x^{u_1}.
$$
\quad Therefore, $x^a=x_1^{a_1}x^{\epsilon_\lambda}x^{u_1}\in I$ because
$a_1\geq q_t\geq q_\lambda$, a contradiction. This proves the claim that
$x_1^{\ell_t}x_2^{a_2}\cdots x_n^{a_n}\notin J$. Then, $h\geq\ell_t$ by
the choice of $h$. From Eq.~\eqref{nov8-25}, we get $x_kx^b=
x_1^hx^{\epsilon_t}x^{\delta_1}$ and since $h\geq\ell_t$, one has
$x_kx^b\in J$. Thus, $x_k\in(J\colon x^b)$. This proves the inclusion
$\mathfrak{p}\subset(J\colon x^b)$. 

To show the inclusion $\mathfrak{p}\supset(J\colon x^b)$, take
$x^c\in(J\colon x^b)$ and recall that $x_1\in\mathfrak{p}$. 
If $\mathfrak{p}$ is the maximal ideal $\mathfrak{m}=(x_1,\ldots,x_n)$, then
$\mathfrak{p}=\mathfrak{m}\subset(J\colon x^b)\subset\mathfrak{m}$ and
$\mathfrak{m}\in{\rm Ass}(J)$. Thus, we may assume
that $\mathfrak{p}\subsetneq\mathfrak{m}$. We argue
by contradiction assuming that $x_k$ does not divide $x^c$ for all
$x_k\in\mathfrak{p}$, that is,
$x^c=\prod_{x_i\notin\mathfrak{p}}x_i^{c_i}$. The product is taken
over a nonempty set because $\mathfrak{p}\subsetneq\mathfrak{m}$.
Since $x^cx^b\in J$, there are two cases to
consider. First we assume that
$x^cx^b=x_1^jx^{\gamma_{i,j}}x^{\theta_2}$. Then,
$$
x_1^{a_1-h}x^cx^b=x^cx^a=x_1^{a_1-h}x_1^jx^{\gamma_{i,j}}x^{\theta_2}
$$
and $x^cx^a\in I$. Thus, $x^c\in(I\colon x^a)=\mathfrak{p}$, a
contradiction. This means that this case cannot occur. Hence, we can
write $x^cx^b=x_1^{\ell_t}x^{\epsilon_t}x^{\delta_1}$ for some $t$.
Multiplying by $x_1^{a_1-h}$, one has 
$$
x^ax^c=x_1^{a_1-h}x^cx^b=x_1^{a_1-h}x_1^{\ell_t}x^{\epsilon_t}x^{\delta_1}=
x_1^{a_1-h+\ell_t}x^{\epsilon_t}x^{\delta_1}.
$$
\quad Hence, as $x^cx^a\notin I$ because $x^c\notin(I\colon
x^a)=\mathfrak{p}$, we get that $a_1-h+\ell_t<q_t$, otherwise
$x^cx^a\in I$. Therefore, recalling that from Eq.~\eqref{nov7-25} we
have $a_1=q_{i_1}-1$, one obtains
\begin{align*}
&q_1-p-1=q_t-\ell_t>a_1-h=(q_{i_1}-1)-h=q_{i_1}-1-h=(\ell_{i_1}+q_1-p-1)-1-h,
\end{align*}
and cancelling out $q_1-p-1$ from the ends, we obtain that
$0>\ell_{i_1}-1-h$, that is, $h\geq \ell_{i_1}$. Then, by the choice 
of $h$, one has $x_1^{\ell_{i_1}}x_2^{a_2}\cdots x_n^{a_n}\notin J$.
From the equality $x^a=x_1^{q_{i_1}-1}x^{\epsilon_{i_1}}x^{\theta_1}$
of  Eq.~\eqref{nov7-25} it follows that $x_2^{a_2}\cdots
x_n^{a_n}=x^{\epsilon_{i_1}}x^w$ for some $x^w$ and consequently 
$$
x_1^{\ell_{i_1}}x_2^{a_2}\cdots
x_n^{a_n}=x_1^{\ell_{i_1}}x^{\epsilon_{i_1}}x^w\in J,
$$
a contradiction. Then, $x^c\in\mathfrak{p}$, and we have the 
inclusion $\mathfrak{p}\supset(J\colon x^b)$.

(B) To show the inclusion ${\rm Ass}(I)\supset{\rm Ass}(J)$, take
$\mathfrak{p}\in {\rm Ass}(J)$, that is, $(J\colon x^b)=\mathfrak{p}$
for some $x^b$, where $b=(b_1,\ldots,b_n)$. Then, $x^b\notin J$ and
$x^b\notin I$ because $I\subset J$. 

$(\mathrm{B}_1)$ $x_1\in\mathfrak{p}$. One has, $x_1x^b\in J$ and,  by
Eq.~\eqref{dec2-25-2}, we can write $x_1x^b$ as:
$$
x_1x^b=
\begin{cases}
x_1^jx^{\gamma_{i,j}}x^\theta& \mbox{
for some }0\leq j\leq p\ \mbox{ or}\\
x_1^{\ell_j}x^{\epsilon_j}x^\theta&\mbox{
for some }1\leq j\leq s,
\end{cases}
$$
where $\ell_j=q_j-q_1+p+1$. Note that $x_1\notin{\rm supp}(x^\theta)$ because $x^b\notin J$.

$(\mathrm{B}_{1.1})$ Let $x_1x^b=x_1^jx^{\gamma_{i,j}}x^\theta$. We will
prove the equality $(I\colon
x^b)=\mathfrak{p}$. One has $x_1\in(I\colon x^b)$ and $b_1=j-1\leq p-1$. Take
$x_k\in\mathfrak{p}$, $k\neq 1$. Then, $x_kx^b\in J$. 
If $x_kx^b=x_1^{j_1}x^{\gamma_{i_1,j_1}}x^{\theta_1}$, then 
$x_k\in(I\colon x^b)$. If
$x_kx^b=x_1^{\ell_{j_1}}x^{\epsilon_{j_1}}x^{\theta_1}$, then $b_1\geq
\ell_{j_1}\geq p+1$, a contradiction because $b_1\leq p-1$. Hence, one
has $\mathfrak{p}\subset(I\colon x^b)$. To show the inclusion
$\mathfrak{p}\supset(I\colon x^b)$, take $x^c\in(I\colon x^b)$. Then, 
$x^cx^b\in I\subset J$ and $x^c\in(J\colon x^b)=\mathfrak{p}$. Thus, 
$x^c\in\mathfrak{p}$. 
this proves that $(I\colon
x^b)=\mathfrak{p}$ and $\mathfrak{p}\in{\rm Ass}(I)$. 

$(\mathrm{B}_{1.2})$ Let $x_1x^b=x_1^{\ell_j}x^{\epsilon_j}x^\theta$ for some
$1\leq j\leq s$. Then, $x_1\notin{\rm supp}(x^\theta)$ because
$x^b\notin J$, and consequently, $b_1=\ell_j-1\geq p$. We let 
$$   
x^a:=x_1^{b_1+q_1-p-1}x_2^{b_2}\cdots x_n^{b_n}=x_1^{q_1-p-1}x^b.
$$
\quad We claim that $x^a\notin I$. We argue by contradiction assuming
that $x^a\in I$. There are two cases to consider. If
$x^a=x_1^{j_1}x^{\gamma_{i_1,j_1}}x^\delta$ for some $0\leq j_1\leq p$, we
can eliminate $x_1$ from both sides of this equality to obtain 
$x_2^{b_2}\cdots x_n^{b_n}=x^{\gamma_{i_1,j_1}}x^{\delta_1}$ and 
$$
x^b=x_1^{b_1}(x^{\gamma_{i_1,j_1}}x^{\delta_1}).
$$
\quad Hence, recalling that $b_1=\ell_j-1\geq p$ and $0\leq j_1\leq
p$, we get 
$$
x^b=x_1^{p}(x^{\gamma_{i_1,j_1}}x^{\delta_2})=(x_1^{j_1}x^{\gamma_{i_1,j_1}})x^{\delta_3}\in
I,
$$
a contradiction. The other case to consider is 
$x^a=x_1^{q_i}x^{\epsilon_i}x^\delta$ for some $1\leq i\leq s$. Then, 
$$
b_1+(q_1-p-1)=b_1+(q_i-\ell_i)\geq q_i,\mbox{ and }b_1\geq \ell_i\geq
p+1.
$$
\quad Eliminating $x_1$ from both sides of the equality
$x^a=x_1^{q_i}x^{\epsilon_i}x^\delta$, we obtain 
$x_2^{b_2}\cdots x_n^{b_n}=x^{\epsilon_i}x^{\delta_1}$. 
Thus,
$x^b=x_1^{b_1}(x^{\epsilon_i}x^{\delta_1})=x_1^{\ell_i}x^{\epsilon_i}x^{\delta_2}\in
J$ because $b_1\geq\ell_i$, a contradiction. This shows that
$x^a\notin I$.

We now prove that $(I\colon x^a)=\mathfrak{p}$. From the equalities:
$$
x_1x^a=(x_1^{b_1+1}x_2^{b_2}\cdots
x_n^{b_n})(x_1^{q_1-p-1})=(x_1x^b)(x_1^{q_j-\ell_j})=
(x_1^{\ell_j}x^{\epsilon_j}x^\theta)(x_1^{q_j-\ell_j})=
x_1^{q_j}x^{\epsilon_j}x^\theta,
$$
we get $x_1x^a\in I$ and $x_1\in(I\colon x^a)$. Now take
$x_k\in\mathfrak{p}$, $k\neq 1$. Then, $x_kx^b\in J$ because
$\mathfrak{p}=(J\colon x^b)$. There are two cases to consider. If
$x_kx^b=x_1^{j_1}x^{\gamma_{i_1,j_1}}x^{\delta}$, $0\leq j_1\leq p$, 
then $x_kx^a\in I$ because $x_kx^b\in I$ and $a\geq b$. Thus, $x_k\in
(I\colon x^a)$. If $x_kx^b=x_1^{\ell_i}x^{\epsilon_i}x^{\delta}$, then
$$
x_kx^a=x_1^{q_1-p-1}(x_kx^b)=x_1^{q_i-\ell_i}(x_1^{\ell_i}x^{\epsilon_i}x^{\delta})=
x_1^{q_i}x^{\epsilon_i}x^{\delta}
$$
and $x_kx^a\in I$. Thus, $x_k\in(I\colon x^a)$. This proves that 
$\mathfrak{p}\subset(I\colon x^a)$. 

To show the inclusion $(I\colon x^a)\subset\mathfrak{p}$, take
$x^c\in(I\colon x^a)$, i.e., $x^cx^a\in I$. If $x_1\in{\rm
supp}(x^c)$, then $x^c\in\mathfrak{p}$ because $x_1\in\mathfrak{p}$.
Thus, we may assume that $x_1\notin{\rm supp}(x^c)$. If
$x^cx^a=x_1^{j_1}x^{\gamma_{i_1,j_1}}x^{\theta}$, $0\leq j_1\leq p$, we
can eliminate $x_1$ from both sides of this equality and obtain
$$
x^cx_2^{b_2}\cdots x_n^{b_n}=x^{\gamma_{i_1,j_1}}x^{\theta_1}\ 
\mbox{ and }\ x^cx^b=x_1^{b_1}x^{\gamma_{i_1,j_1}}x^{\theta_1}.
$$
\quad Then, $x^cx^b\in I\subset J$ because $b_1\geq p\geq j_1\geq 0$.
Consequently, $x^c\in(J\colon x^b)=\mathfrak{p}$ and
$x^c\in\mathfrak{p}$. If $x^cx^a=x_1^{q_i}x^{\epsilon_i}x^{\theta}$,
then $b_1+q_1-p-1\geq q_i$ and since $q_1-p-1=q_i-\ell_i$, one has
$b_1\geq\ell_i$. Eliminating $x_1$ from both sides of the
equality above, one has 
$x^cx_2^{b_2}\cdots x_n^{b_n}=x^{\epsilon_i}x^{\theta_1}$. Hence,  
$x^cx^b=x_1^{\ell_i}x^{\epsilon_i}x^{\theta_2}$ because
$b_1\geq\ell_i$. Consequently, $x^c\in(J\colon x^b)=\mathfrak{p}$ and
$x^c\in\mathfrak{p}$. This proves that $(I\colon
x^a)\subset\mathfrak{p}$. Therefore, $(I\colon x^a)=\mathfrak{p}$ 
and $\mathfrak{p}\in{\rm Ass}(I)$. 

$(\mathrm{B}_2)$ $x_1\notin\mathfrak{p}$.  Take
$x_k\in\mathfrak{p}=(J\colon x^b)$.
Then, $x_kx^b\in J$ and, from Eq.~\eqref{dec2-25-2}, we can write 
$$
x_kx^b=
\begin{cases}
x_1^jx^{\gamma_{i,j}}x^\theta& \mbox{
for some }0\leq j\leq p\ \mbox{ or}\\
x_1^{\ell_j}x^{\epsilon_j}x^\theta&\mbox{
for some }1\leq j\leq s.
\end{cases}
$$
\quad Note that $x_k\notin{\rm supp}(x^\theta)$ because $x^b\notin J$.
Recalling that $q_1\leq q_2\leq\cdots\leq q_s$, we let 
$$
x^a:=x_1^{q_s+b_1}x_2^{b_2}\cdots x_n^{b_n}=x_1^{q_s}x^b. 
$$
\quad We claim that $x^a\notin I$. Indeed, if $x^a=x_1^{q_s}x^b\in
I\subset J$, then $x_1^{q_s}\in(J\colon x^b)=\mathfrak{p}$ and
$x_1\in\mathfrak{p}$, a contradiction.

$(\mathrm{B}_{2.1})$ Let $x_kx^b=x_1^jx^{\gamma_{i,j}}x^\theta$.
Cancelling out $x_1$ from both sides of this equality, one obtains 
the equality $x_kx_2^{b_2}\cdots
x_n^{b_n}=x^{\gamma_{i,j}}x^{\theta_1}$. Hence, 
$$
x_kx^a=x_k(x_1^{q_s+b_1}x_2^{b_2}\cdots
x_n^{b_n})=x_1^{q_s+b_1}x^{\gamma_{i,j}}x^{\theta_1},
$$
and since $q_s\geq\ell_s=q_s-q_1+p+1\geq p+1$, we get that $x_kx^a\in
I$ and $x_k\in(I\colon x^a)$.

$(\mathrm{B}_{2.2})$ Let $x_kx^b=x_1^{\ell_j}x^{\epsilon_j}x^\theta$. 
Cancelling out $x_1$ from both sides of this equality, one obtains 
the equality $x_kx_2^{b_2}\cdots
x_n^{b_n}=x^{\epsilon_j}x^{\theta_1}$. Hence, 
$$
x_kx^a=x_k(x_1^{q_s+b_1}x_2^{b_2}\cdots
x_n^{b_n})=x_1^{q_s+b_1}x^{\epsilon_j}x^{\theta_1},
$$
and since 
$q_s\geq q_j$, 
we get that $x_kx^a\in
I$ and $x_k\in(I\colon x^a)$.

Therefore, from $(\mathrm{B}_{2.1})$ and $(\mathrm{B}_{2.2})$, one has 
$\mathfrak{p}\subset(I\colon x^a)$. To show the inclusion
$\mathfrak{p}\supset(I\colon x^a)$, take $x^c\in(I\colon x^a)$. Then,  
$x^cx^a=x^c(x_1^{q_s}x^b)\in I\subset J$. Thus,  
$x^cx_1^{q_s}\in(J\colon x^b)=\mathfrak{p}$ and since
$x_1\notin\mathfrak{p}$, one has $x^c\in\mathfrak{p}$. This proves that
$(I\colon x^a)=\mathfrak{p}$ 
and $\mathfrak{p}\in{\rm Ass}(I)$.

(ii) This part follows from the proof of part (i) and
Proposition~\ref{carlitos-vila}. Indeed, pick
$\mathfrak{p}\in{\rm Ass}(I)$ and $x^a$ such that $(I\colon
x^a)=\mathfrak{p}$ and $\deg(x^a)={\rm v}(I)$. By the proof of part
(i) there is $x^b$ such that $(J\colon x^b)=\mathfrak{p}$ and
$\deg(x^b)\leq\deg(x^a)$. Hence, by Proposition~\ref{carlitos-vila},
 ${\rm v}(J)\leq {\rm v}(I)$. 

(iii) This follows at once from part (i).
\end{proof}

\begin{corollary}\label{ass-v-number-coro} 
{\rm(A)} If $\mathfrak{p}=(I\colon x^a)\in{\rm Ass}(I)$, $x^a=x_1^{a_1}\cdots
x_n^{a_n}$, and
$J=I_{sft}$, then  
$$
\mathfrak{p}=
\begin{cases}
(J\colon x^a)& \mbox{
if }x^a\notin J,\\
(J\colon x_1^hx_2^{a_2}\cdots x_n^{a_n})&\mbox{
if }x^a\in J \mbox{ and }h=\max\{i\mid x_1^ix_2^{a_2}\cdots
x_n^{a_n}\notin J\}.
\end{cases}
$$
\quad {\rm(B)} If $\mathfrak{p}=(J\colon x^b)\in{\rm Ass}(J)$, $x^b=x_1^{b_1}\cdots
x_n^{b_n}$, and $J=I_{sft}$, then  
$$
\mathfrak{p}=
\begin{cases}
(I\colon x^b)& \mbox{
if }x_1\in\mathfrak{p}\mbox{ and }x_1x^b=x_1^jx^{\gamma_{i,j}}x^\theta,\\
(I\colon x_1^{q_1-p-1}x^b)&\mbox{
if }x_1\in\mathfrak{p}\mbox{ and
}x_1x^b=x_1^{\ell_j}x^{\epsilon_j}x^\theta,
\\
(I\colon x_1^{q_s}x^b)&\mbox{
if }x_1\notin\mathfrak{p}.
\end{cases}
$$
\end{corollary}
\begin{proof} This follows from the proof of
Theorem~\ref{ass-v-number}(i).
\end{proof}

\section{Examples: Illustrating the shift and polarization
operations}\label{section-examples} 

In this section, in Example~\ref{signature-example}, we illustrate 
how the procedure given in Section~\ref{tight-matrix-section} and
coded in Procedure~\ref{procedure-signature} works to obtain 
signatures through recursive steps called shift operations. 
The algebraic invariants in the following examples are computed with 
\textit{Macaulay}$2$ \cite{mac2}.

\begin{example}\label{signature-principal-ideal}\rm Let $R=K[x_1,x_2,x_3]$ be
a polynomial ring and let $I=(x_1x_2x_3)$. The incidence
matrices of $I$ and 
its signature ${\rm sgn}(I)$ are given by 
\begin{align*}
A=\left[\begin{matrix}
1\\
1\\
1
\end{matrix}\right]&\ \mbox{ and }\ {\rm sgn}(A)=\left[\begin{matrix}
0\\
0\\
0
\end{matrix}\right],
\end{align*}
respectively. Then ${\rm sgn}(I)=(1)=R$. 
\end{example}

\begin{example}\label{signature-height-one}\rm Let $R=K[x_1,x_2,x_3]$ be
a polynomial ring 
and let $I=(x_1x_2x_3,x_1x_2^2)$. The incidence
matrices of $I$ and 
its signature ${\rm sgn}(I)$ are given by 
\begin{align*}
A=\left[\begin{matrix}
1&1\\
1&2\\
1&0
\end{matrix}\right]&\ \mbox{ and }\ {\rm sgn}(A)=\left[\begin{matrix}
0&0\\
0&1\\
1&0
\end{matrix}\right],
\end{align*}
respectively. Then ${\rm sgn}(I)=(x_3,x_2)$, ${\rm ht}(I)=1$ and 
${\rm Ass}(I)\neq{\rm Ass}({\rm sgn}(I))$. Thus, 
for monomial ideals of height $1$, the associated primes of
the ideal and its signature may be different. 
\end{example}

\begin{example}\label{signature-not-anti-chain}\rm Let $R=K[x_1,x_2,x_3]$ be
a polynomial ring 
and let $I=(x_2x_3^2,x_1x_3,x_1^2)$ and $J=(x_2x_3,x_1x_3,x_1^2)$.
Then, $I$ and $J$ are signature ideals and $I\subsetneq J$. Thus, 
the class of signature ideals in $n$ variables and $q$ minimal
generators do not form an anti-chain under inclusion. A result of
Maclagan \cite{maclagan} shows that all antichains are finite in the
poset of monomial ideals in the polynomial ring $R$, 
ordered by inclusion.  
\end{example}

\begin{example}\label{signature-one-gap}\rm Let $R=K[x_1,x_2,x_3]$ be
a polynomial ring 
and let $I=(x_2^3,\, x_1x_2^2,\,x_1^3x_3^2,\,x_1^4x_2x_3)$. Following
the notation introduced in Section~\ref{tight-matrix-section}, $I$ has
a gap with respect to $x_1$ occurring at 
$x_1x_2^2,\,x_1^3x_3^2$, or simply $I$ has
a gap in $x_1^3x_3^2$ at $x_1$, and one has $p=1$, $q_1=3$, $q_2=4$, $q_1-p=2$, 
$$ 
I_{\rm pol}=(x_2^3,\,x_1x_2^2,\,x_0x_1x_3^2,\,x_0x_1^2x_2x_3),\ \mbox{
and }\ 
I_{\rm sft}=(x_2^3,\,x_1x_2^2,\,x_1^2x_3^2,\,x_1^3x_2x_3).
$$
\quad In this case $I_{\rm sft}={\rm sgn}(I)$ and the incidence
matrices of $I$ and 
its signature are given by 
\begin{align*}
A=\left[\begin{matrix}
0&1&3&4\\
3&2&0&1\\
0&0&2&1
\end{matrix}\right]&\ \mbox{ and }\ {\rm sgn}(A)=\left[\begin{matrix}
0&1&2&3\\
3&2&0&1\\
0&0&2&1
\end{matrix}\right],
\end{align*}
respectively.  
\end{example}

\begin{example}\label{signature-example}\rm Let $R=K[x_1,x_2,x_3,x_4]$ and let  
$I=(x_1x_2x_3^2,\,x_1^6x_3^7,\,x_2^3x_4,\,x_2x_3^3x_4,\,x_2x_4^3)$. 
The depth of $R/I$ is $0$ and the incidence matrix of $I$ and its signature are given by 
\begin{align*}
A=\left[\begin{matrix}
1&6&0&0&0\\
1&0&3&1&1\\
2&7&0&3&0\\
0&0&1&1&3  
\end{matrix}\right]&\ \mbox{ and }\ {\rm sgn}(A)=\left[\begin{matrix}
1&2&0&0&0\\
1&0&2&1&1\\
1&3&0&2&0\\
0&0&1&1&2  
\end{matrix}\right],
\end{align*}
respectively, and ${\rm sgn}(I)=(x_1x_2x_3
,\,x_1^2x_3^3,\,x_2^2x_4,\,x_2x_3^2x_4,\,x_2x_4^2)$. For convenience
we illustrate the recursive procedure to obtain the signature of $I$ by 
successively applying the operation $I\mapsto I_{\rm
sft}$ (see Proposition~\ref{carlitos-vila} and
Lemma~\ref{depth-dim-inv}). The procedure gives the signature ${\rm
sgn}(I)$ of $I$ whose quotient ring $R/{\rm sgn}(I)$ has the same
depth than $R/I$ (Theorem~\ref{depth=sgn}). 

\begin{enumerate} 

\item Generators of $I$ in non-decreasing order with respect to  $x_1$:

\smallskip

$G(I)=\{x_2^3x_4, x_2x_3^3x_4, x_2x_4^3, x_1x_2x_3^2, x_1^6x_3^7\}$.

$p=1$, $q_1=6$, $q_1-p=5$, $x_0\rightarrow x_1^5$,

$G(I_{\rm pol})=\{x_2^3x_4, x_2x_3^3x_4, x_2x_4^3, x_1x_2x_3^2,
x_0x_1x_3^7\}$,

$G(I_{\rm sft})=\{x_2^3x_4, x_2x_3^3x_4, x_2x_4^3, x_1x_2x_3^2,
x_1^2x_3^7\}$.

${\rm reg}(R/I)=12$, ${\rm reg}(R/I_{\rm
sft})={\rm reg}(R[x_0]/I_{\rm pol})=8$, v$(I)=4$, v($I_{\rm sft})=4$.

Since the first row of the incidence matrix of $I_{\rm sft}$ is tight
we now move to the variable $x_2$ and rename $I_{\rm sft}$ by $I$.

\item Generators of $I$ in non-decreasing order with respect to  $x_2$:

\smallskip

$G(I)=\{x_1^2x_3^7, x_2x_3^3x_4,  x_2x_4^3, x_1x_2x_3^2, x_2^3x_4\}$.

$p=1$, $q_1=3$, $q_1-p=2$, $x_0\rightarrow x_2^2$,

$G(I_{\rm pol})=\{x_1^2x_3^7, x_2x_3^3x_4,  x_2x_4^3, x_1x_2x_3^2,
x_0x_2x_4\}$,

$G(I_{\rm sft})=\{x_1^2x_3^7, x_2x_3^3x_4,  x_2x_4^3, x_1x_2x_3^2,
x_2^2x_4\}$.

${\rm reg}(R/I)=8$, ${\rm reg}(R/I_{\rm
sft})={\rm reg}(R[x_0]/I_{\rm pol})=8$, v$(I)=4$, v($I_{\rm sft})=4$.

Since the 1st and 2nd row of the incidence matrix of $I_{\rm sft}$ are tight
we now move to the variable $x_3$ and rename $I_{\rm sft}$ by $I$.

\item Generators of $I$ in non-decreasing order with respect to  $x_3$:

\smallskip

$G(I)=\{x_2x_4^3, x_2^2x_4,   x_1x_2x_3^2, x_2x_3^3x_4,
x_1^2x_3^7\}$.

$p=0$, $q_1=2$, $q_2=3$, $q_3=7$, $q_1-p=2$, $x_0\rightarrow x_3^2$,

$G(I_{\rm pol})=\{x_2x_4^3, x_2^2x_4,   x_1x_2x_0, x_2x_0x_3x_4,
x_1^2x_0x_3^5\}$,

$G(I_{\rm sft})=\{x_2x_4^3, x_2^2x_4,   x_1x_2x_3, x_2x_3^2x_4,
x_1^2x_3^6\}$.

${\rm reg}(R/I)=8$, ${\rm reg}(R/I_{\rm
sft})={\rm reg}(R[x_0]/I_{\rm pol})=7$, v$(I)=4$, v($I_{\rm sft})=3$.

Since the 3rd row of the incidence matrix of $I_{\rm sft}$ not yet tight
we need to continue to do one more step with variable $x_3$ and
rename $I_{\rm sft}$ by $I$.

\item Generators of $I$ in non-decreasing order with respect to  $x_3$:

\smallskip

$G(I)=\{x_2x_4^3, x_2^2x_4,   x_1x_2x_3, x_2x_3^2x_4,
x_1^2x_3^6\}$.

$p=2$, $q_1=6$, $q_1-p=4$, $x_0\rightarrow x_3^4$,

$G(I_{\rm pol})=\{x_2x_4^3, x_2^2x_4,   x_1x_2x_3, x_2x_3^2x_4,
x_1^2x_0x_3^2\}$,

$G(I_{\rm sft})=\{x_2x_4^3, x_2^2x_4,   x_1x_2x_3, x_2x_3^2x_4,
x_1^2x_3^3\}$.

${\rm reg}(R/I)=7$, ${\rm reg}(R/I_{\rm
sft})={\rm reg}(R[x_0]/I_{\rm pol})=4$, v$(I)=3$, v($I_{\rm sft})=3$.

Since the 1st, 2nd and 3rd row of the incidence matrix of $I_{\rm sft}$ are tight
we now move to the variable $x_4$ and rename $I_{\rm sft}$ by $I$.

\item Generators of $I$ in non-decreasing order with respect to  $x_4$:

\smallskip

$G(I)=\{x_1x_2x_3, x_1^2x_3^3, x_2^2x_4,x_2x_3^2x_4, x_2x_4^3
\}$.

$p=1$, $q_1=3$, $q_1-p=2$, $x_0\rightarrow x_4^2$,

$G(I_{\rm pol})=\{x_1x_2x_3, x_1^2x_3^3, x_2^2x_4,x_2x_3^2x_4,
x_2x_0x_4
\}$,

$G(I_{\rm sft})=\{x_1x_2x_3, x_1^2x_3^3, x_2^2x_4,x_2x_3^2x_4,
x_2x_4^2
\}$.

${\rm reg}(R/I)=4$, ${\rm reg}(R/I_{\rm
sft})={\rm reg}(R[x_0]/I_{\rm pol})=4$, v$(I)=3$, v($I_{\rm sft})=3$.
\end{enumerate}

Since all rows of the incidence matrix of $I_{\rm sft}$ are tight
the process stops here and we obtain that $I_{\rm sft}={\rm sgn}(I)$.

Now we show how to recursively apply the weighted partial polarization
operation $I\mapsto I_{\rm
pol}$ to obtain the \textit{full weighted polarization} $I_{\rm pol}$ of $I$, 
the quotient ring $M=S/I_{\rm pol}$ and the
$M$-regular sequences 
$\underline{f}$ and
$\underline{g}$ of Proposition~\ref{sep22-25} such that 
$M/(\underline{f})=R/I$ and $M/(\underline{g})=R/{\rm
sgn}(I)$. 

From (1)--(5), let $z_1,\ldots,z_r$, $r=5$, be new variables, one variable for
each step of the process above, let $S=R[z_1,\ldots,z_5]$, and let 
\begin{align*}
f_1=z_1-x_1^5,\ f_2=z_2-x_2^2,\ f_3=z_3-x_3^2,\ f_4=z_4-x_3^4,\
f_5=z_5-x_4^2,\\
g_1=z_1-x_1,\ g_2=z_2-x_2,\ g_3=z_3-x_3,\ g_4=z_4-x_3,\
g_5=z_5-x_4.
\end{align*}   
\quad Following (1)--(5), we obtain that a full weighted
polarization of $I$ is given by:
$$
I_{\rm pol}=(z_2x_2x_4,\  x_1x_2z_3,\  x_2z_3x_3x_4,\
z_1x_1z_3z_4x_3,\  x_2x_4z_5).
$$
The algebraic invariants if $I$, $I_{\rm pol}$, and ${\rm sgn}(I)$ are given by: 
\begin{align*}
&{\rm ht}(I)={\rm ht}(I_{\rm pol})={\rm ht}({\rm sgn}(I))=2,\ 
\dim(M)=\dim(S/I_{\rm pol})=\dim(R/I)+r=7,\\
&{\rm depth}(R/{\rm sgn}(I))={\rm depth}(R/I)=0,\ {\rm pd}(R/I)={\rm
pd}(R/{\rm sgn}(I))=4,\ \\
&{\rm depth}(M)={\rm depth}(S/I_{\rm pol})={\rm depth}(R/I)+r=5,\
\mbox{ and }\ {\rm pd}(S/I_{\rm pol})=4,\\
&
{\rm reg}(R/I)=12,\ {\rm reg}(R/{\rm sgn}(I))={\rm reg}(S/I_{\rm pol})=4,\ {\rm v}(I)=4,\ {\rm
v}({\rm sgn}(I))=3.\\
&\mbox{Letting }\deg(z_1)=5,\, \deg(z_2)=2,\, \deg(z_3)=2,\, \deg(z_4)=4,\,
\deg(z_5)=2, \mbox{ one has }\\
&{\rm reg}_R(S/(I_{\rm pol},z_1-x_1^5,\,
z_2-x_2^2,\, z_3-x_3^2,\, z_4-x_3^4,\,
z_5-x_4^2))=12,\mbox{ and if the regularity}\\ 
&\mbox{is taken over }S,\mbox{ then }
{\rm reg}_S(S/(I_{\rm pol},z_1-x_1^5,\,
z_2-x_2^2,\, z_3-x_3^2,\, z_4-x_3^4,\,
z_5-x_4^2))=22.
\end{align*}
\end{example}

\begin{example}\label{example-ass-ij}
Let $I$ and $J=I_{\rm sft}$ be the ideals in step (1) 
of Example~\ref{signature-example} given by:
$$
I=(x_2^3x_4, x_2x_3^3x_4, x_2x_4^3, x_1x_2x_3^2, x_1^6x_3^7),\,\  
J=(x_2^3x_4, x_2x_3^3x_4, x_2x_4^3, x_1x_2x_3^2,
x_1^2x_3^7),
$$
and let $p=1$, $q_1=6$, $q_1-p=5$, and $s=1$. 
Given an associated prime $\mathfrak{p}=(I\colon x^a)$ of $I$, 
by Corollary~\ref{ass-v-number-coro}, one can find $x^b$ such that 
$\mathfrak{p}=(J\colon x^b)$, and one has:

$(x_1,x_2)=(I\colon x_1^5x_3^7)=(J\colon  x_1x_3^7)$,

$(x_1,x_4)=(I\colon x_2x_3^3)=(J\colon  x_2x_3^3)$,

$(x_2,x_3)=(I\colon x_1^6x_3^6)=(J\colon  x_1^6x_3^6)$,

$(x_3,x_4)=(I\colon x_1x_2^3x_3)=(J\colon  x_1x_2^3x_3)$,

$(x_2,x_3,x_4)=(I\colon x_1x_2^2x_3x_4^2)=(J\colon  x_1x_2^2x_3x_4^2)$,

$(x_1,x_2,x_3,x_4)=(I\colon x_2^2x_3^2x_4^2)=(J\colon x_2^2x_3^2x_4^2)$.

Conversely, given an associated prime $\mathfrak{p}=(J\colon x^b)$ of
$J$, by Corollary~\ref{ass-v-number-coro}, one can find $x^a$ such that 
$\mathfrak{p}=(I\colon x^a)$, and one has:

$(x_1,x_2)=(J\colon  x_1x_3^7)=(I\colon x_1^5x_3^7)$,

$(x_1,x_4)=(J\colon  x_2x_3^3)=(I\colon x_2x_3^3)$,

$(x_2,x_3)=(J\colon  x_1^2x_3^6)=(I\colon x_1^8x_3^6)$,

$(x_3,x_4)=(J\colon  x_1x_2^3x_3)=(I\colon x_1^7x_2^3x_3)$,

$(x_2,x_3,x_4)=(J\colon  x_1x_2^2x_3x_4^2)=(I\colon x_1^7x_2^2x_3x_4^2)$,

$(x_1,x_2,x_3,x_4)=(J\colon  x_2^2x_3^2x_4^2)=(I\colon x_2^2x_3^2x_4^2).$ 
\end{example}

\begin{example}\label{alldepth}
Let $R=K[x_1,x_2,x_3,x_4]$ be a polynomial ring over a field $K$ and
let $I$ be the monomial ideal $(x_1^2x_3,\, x_2^2x_3,\,x_1^2x_4)$.
Then,
\begin{align*}
&{\rm reg}(R/I)=
{\rm reg}(R/(I\colon x_4))={\rm reg}(R/(I,x_4))=3,\\
&{\rm sgn}(I)=(x_1x_3,\, x_2x_3,\, x_1x_4),\ {\rm reg}(R/{\rm
sgn}(I))=1.
\end{align*}  
Thus, the regularity is not necessarily preserved under the signature.
\end{example}

\begin{example}\label{q-n-infinity}
Given integers $n\geq 2$ and $q\geq 1$, there are distinct monomials 
$x_1^q,x^{a_2},\ldots,x^{a_q}$ of degree $q$ in the polynomial ring
$R=K[x_1,\ldots,x_n]$; because the
number of monomials of $R$ of degree $q$ is $\binom{q+n-1}{n-1}$ and
$\binom{q+n-1}{n-1}\geq q+1$. Letting
$L_k:=(x_1^{kq},x^{a_2},\ldots,x^{a_q})$ for $k\geq 1$, gives an
infinite family of monomial ideals of $R$ minimally generated by $q$
monomials. 
\end{example}

\begin{example}\label{reviewer}
Let $R=K[x_1,x_2,x_3]$ be a polynomial ring over a field $K$ and
let $I$ and $J$ be the monomial ideals
$$
I=(x_1x_3^3,\, x_2x_3^3)\ \mbox{ and }\
J=(x_1x_3^5,\, x_2x_3^5,\, x_1x_2^2),
$$
respectively. These ideals are edge ideals of weighted oriented
graphs. Then,
\begin{align*}
{\rm sgn}(I)=(x_1,x_2)\ \mbox{ and }\  
{\rm sgn}(J)=(x_1x_3,\, x_2x_3,\, x_1x_2^2).
\end{align*}  
\quad If $L=(x_1x_2,\, x_2x_3)$, then $L$ is squarefree, ${\rm
ht}(L)=1$,  
and ${\rm sgn}(L)=(x_1,\, x_3)\neq L$.
\end{example}

\section*{Acknowledgments.} 
\textit{Macaulay}$2$ \cite{mac2} was used to implement 
algorithms for computing the signature of monomial ideals, and 
to verify the Cohen--Macaulay property of a given list 
of monomial ideals.    
We thank the referee for a careful 
reading of the paper and for the improvements suggested. 

\begin{appendix}

\section{Procedures for \textit{Macaulay}$2$}\label{AppendixA}

\begin{procedure}\label{procedure-signature}\rm
This procedure computes the signature of a monomial ideal and the
signature of its incidence matrix using \textit{Macaulay}$2$
\cite{mac2}. This procedure corresponds to 
Example~\ref{signature-example}. 
\begin{verbatim}
restart
loadPackage("Normaliz",Reload=>true)
load "SymbolicPowers.m2"
R=QQ[x1,x2,x3,x4]
I=monomialIdeal(x1*x2*x3^2,x1^6*x3^7,x2^3*x4,x2*x3^3*x4,x2*x4^3)
--exponent vectors of the minimal generators of I
B = matrix flatten apply(flatten entries gens I, exponents)
--incidence matrix of I
A=transpose B
--rows of B
l= entries B
--rows of the incidence matrix of I that are used 
--to compute the signature 
lt=entries A
--This is the signature of A up to permutation of columns
sgnA=matrix apply(lt, y->flatten apply(y, x-> 
positions(sort toList set(y), i -> i == x)) ) 
--This is the signature sgn(I) of I
sgnI=monomialIdeal(for i in entries transpose sgnA list R_i)
\end{verbatim}  
\end{procedure}

\begin{procedure}\label{procedure-list}
Given a set of incidence matrices, the following procedure 
for \textit{Macaulay}$2$ \cite{mac2} compute their signatures and the
signature of the corresponding 
monomial ideals. Then, it decides which elements of the list 
correspond to Cohen--Macaulay (C-M) monomial ideals and put them on a
list in LaTex format. The input for
this procedure 
is a list of incidence matrices of monomial ideals. The list that we
use below as input, is the set of all $3\times 3$ signature matrices,
and the output is the list of all Cohen--Macaulay 
signature ideals in $3$ variables and generated by $3$ monomials 
shown in part (a) of Appendix~\ref{AppendixB}. This procedure also
determines all Gorenstein monomial ideals of the given list and put
them on a new list. To compute other
examples simply change List3x3 by any other list.

\begin{verbatim}
restart
loadPackage "TorAlgebra"
R=QQ[x1,x2,x3]
X=toList flatten entries vars R
--List of signature matrices of size 3x3 
--The procedure below can be applied to any list of 
--incidence matrices of monomial ideals
List3x3=toList{
matrix{{1,0,0},{0,1,0},{0,0,1}},matrix{{0,0 ,1},{2,0,1},{0,2,1}},
matrix{{1,1,0 },{1,0,1 },{0,1,1}},matrix{{1,1,0},{1,0,2},{0,1,2}},
matrix{{0,1,0},{1,0,0},{0,1,2}},matrix{{1,0,0},{1,0,2},{0,2,1}},
matrix{{1,0,1},{1,1,0},{0,1,2}},matrix{{0,1,2},{0,2,1},{2,0,1}},
matrix{{1,0,0},{1,1,0},{0,1,2}},matrix{{1,0,0},{0,2,1},{0,1,2}},
matrix{{0,1,1},{1,2,0},{1,0,2}},matrix{{0,1,2},{1,0,2},{1,2,0}},
matrix{{0,0,1},{1,0,2},{1,2,0}},matrix{{0,1,1},{0,2,1},{2,0,1}},
matrix{{1,1,0 },{0,2,1 },{1,0,2}},matrix{{1,2,0},{2,0,1},{0,1,2}}}
--List of signature matrices in the right format
List2=apply(List3x3,x->entries x)
--This is the function that gives the signature of an incidence matrix
sgnA=(n)-> matrix apply((List2)#n, 
y->flatten apply(y, x-> positions(sort toList set(y), i -> i == x)))
--This gives all signatures of the original list of matrices
ListSgn=apply(0..#List2-1,x->sgnA(x))
--This is the signature function that gives the 
--signature ideal sgn(I) 
sgnI=(n)->monomialIdeal(for i in entries transpose n list R_i)
--This gives the list of signature ideals of the list of matrices
ListSgnI=toList apply(ListSgn,sgnI)
--This is the function to check the Cohen-Macaulay (C-M) property 
CM=(n)-> pdim coker gens gb n-codim(n) 
--The number of 0's in this list gives the number of C-M signatures
apply(ListSgnI,CM)
apply(ListSgnI,CM)-set{1}
--This is the number of C-M signature of 3x3 matrices
#(apply(ListSgnI,CM)-set{1})
--This gives the list of all C-M monomial ideals in our list
CMI=apply(ListSgnI, n-> 
if (pdim coker gens gb n-codim(n))==0 then n else 1)-set{1}
--This gives the list of all C-M monomial ideals for use in LaTex 
toString (apply(ListSgnI, n-> 
if (pdim coker gens gb n-codim(n))==0 then n else 1)-set{1})
--This gives the number of all C-M monomial ideals in our list
#(apply(ListSgnI, n-> 
if (pdim coker gens gb n-codim(n))==0 then n else 1)-set{1})
--This gives the list of all Gorenstein ideals in List3x3
GI=apply(CMI, n-> if isGorenstein(n)==true then n else 1)
GI=(apply(CMI, n-> if isGorenstein(n)==true then n else 1)-set{1})
\end{verbatim}
\end{procedure}

\section{A List of signature Cohen--Macaulay monomial ideals}
\label{AppendixB}
The full list of signature matrices of sizes $3\times 3$, $4\times 3$, 
and $3\times 4$ are given in \cite[Chapter~4]{moran-tesis}. 
We assume that a monomial ideal with incidence matrix of size $n\times q$ 
has $q$ minimal generators, $n$ variables, and that all variables 
$x_1,\ldots,x_n$ occur in $G(I)$. Using
Procedure~\ref{procedure-list}, we obtain the following three lists of
Cohen--Macaulay ideals:

(a) Out of 16 signature matrices of size $3\times 3$, there are
exactly $10$ of them that correspond to
the following Cohen--Macaulay monomial ideals.

$(x_1,\, x_2,\, x_3)$,\  $(x_1x_2,\, x_1x_3,\, x_2x_3)$, $(x_1x_2,\, x_1x_3,\, x_2^2x_3^2)$,\  $(x_1x_2,\, x_2x_3,\,
x_1x_3^2)$, 

$(x_1x_2,\, x_2x_3,\, x_3^2)$,\  $(x_1x_2^2,\, x_2x_3,\, x_1x_3^2)$,\ 
$(x_1^2x_2^2,\, x_2x_3,\, x_1x_3^2)$,

$(x_1x_2^2,\, x_2x_3,\, x_3^2)$,\  $(x_1x_2^2,\, x_1x_2x_3,\, x_3^2)$,\ 
 $(x_1x_2^2,\, x_1x_3,\, x_2x_3^2)$.

\medskip

(b) Out of 59 signature matrices of size $4\times 3$, there
are exactly 
the following $29$ signature matrices that correspond to Cohen--Macaulay monomial
ideals.

$(x_1x_3,x_2x_4,x_3x_4)$, $(x_1x_3,x_3x_4,x_2x_4^2)$, $(x_1x_3^2,x_3x_4,x_2x_4^2)$,

$(x_1x_2x_3^2,x_2x_3x_4,x_4^2)$, $(x_1x_2x_3,x_2x_4,x_3x_4)$, $(x_1x_2x_3,x_3x_4,x_2x_4^2)$,

$(x_1x_2x_3^2,x_2x_4,x_3x_4)$, $(x_1x_2x_3^2,x_3x_4,x_2x_4^2)$, $(x_1x_2x_3^2,x_2x_4,x_3x_4^2)$, 

$(x_2x_3,x_1x_2x_4,x_3x_4^2)$, 
$(x_1x_2,x_2x_3x_4,x_3x_4^2)$, 
$(x_1x_2x_3,x_2x_4,x_3^2x_4^2)$, 

$(x_1x_2^2x_3^2,x_2x_4,x_3x_4)$, 
$(x_1x_2^2x_3^2,x_3x_4,x_2x_4^2)$, 
$(x_2x_3,x_1x_2^2x_4,x_3x_4^2)$, 

$(x_1x_2^2,x_2x_3x_4,x_3x_4^2)$, 
$(x_1x_2^2x_3,x_2x_4,x_3^2x_4^2)$, 
$(x_1x_2^2x_3,x_1x_2x_4,x_3x_4)$, 

$(x_1x_2^2x_3,x_3x_4,x_1x_2x_4^2)$, 
$(x_1x_2^2x_3,x_1x_2x_4,x_3x_4^2)$, 
$(x_1x_2^2x_3^2,x_1x_2x_4,x_3x_4)$, 

$(x_1x_2^2x_3^2,x_3x_4,x_1x_2x_4^2)$, 
$(x_1x_2^2x_3^2,x_1x_2x_4,x_3x_4^2)$, 
$(x_1x_2x_3,x_1x_2^2x_4,x_3x_4^2)$, 

$(x_1x_2^2,x_1x_2x_3x_4,x_3x_4^2)$, 
$(x_1x_2^2x_3,x_1x_2x_4,x_3^2x_4^2)$, 
$(x_1x_2,x_1x_3^2x_4,x_2x_3x_4^2)$, 

$(x_1x_2^2x_3,x_1x_4,x_2x_3^2x_4^2)$, 
$(x_1^2x_2x_3^2,x_3x_4,x_1x_2^2x_4^2)$.

\medskip 
(c) Out of 285 signature matrices of size $3\times 4$, there
are exactly the following $80$ signature matrices that correspond to Cohen--Macaulay monomial
ideals.

$(x_1,x_2^2,x_2x_3,x_3^2)$, 
$(x_1x_2^2,x_2^2x_3,x_2x_3^2,x_3^3)$, 
$(x_1x_2^3,x_2^2x_3,x_2x_3^2,x_3^3)$, 

$(x_1x_2^2,x_1x_2x_3,x_2x_3^2,x_3^3)$, 
$(x_1x_2^3,x_1x_2^2x_3,x_2x_3^2,x_3^3)$,
$(x_1x_2^2,x_2^2x_3,x_1x_3^2,x_2x_3^2)$, 

$(x_1x_2^2,x_2^2x_3,x_2x_3^2,x_1x_3^3)$, 
$(x_1x_2^3,x_2^2x_3,x_2x_3^2,x_1x_3^3)$, 
$(x_1^2x_2,x_1x_2x_3,x_2x_3^2,x_3^3)$, 

$(x_1^2x_2^2,x_1x_2x_3,x_2x_3^2,x_3^3)$, 
$(x_1^2x_2^2,x_1x_2^2x_3,x_2x_3^2,x_3^3)$,
$(x_1^2x_2^3,x_1x_2^2x_3,x_2x_3^2,x_3^3)$, 

$(x_1^2,x_2^2,x_1x_2x_3,x_3^2)$, 
$(x_1^2x_2^2,x_2^2x_3,x_1x_3^2,x_2x_3^2)$, 
$(x_1^2x_2^2,x_2^2x_3,x_2x_3^2,x_1x_3^3)$, 

$(x_1^2x_2^3,x_2^2x_3,x_1x_3^2,x_2x_3^2)$, 
$(x_1^2x_2^3,x_2^2x_3,x_2x_3^2,x_1x_3^3)$, 
$(x_1x_2^3,x_1x_2^2x_3,x_1x_2x_3^2,x_3^3)$, 

$(x_1x_2^2,x_1x_2x_3,x_1x_3^2,x_2x_3^2)$, 
$(x_1x_2^2,x_1x_2x_3,x_2x_3^2,x_1x_3^3)$, 
$(x_1x_2^2,x_1x_2x_3,x_1x_3^2,x_2x_3^3)$, 

$(x_1x_2^3,x_1x_2^2x_3,x_1x_3^2,x_2x_3^2)$,
$(x_1x_2^3,x_1x_2^2x_3,x_2x_3^2,x_1x_3^3)$,
$(x_1x_2^3,x_1x_2^2x_3,x_1x_3^2,x_2x_3^3)$, 

$(x_1x_2^2,x_1x_2x_3,x_1x_3^2,x_2^2x_3^2)$,
$(x_1x_2^2,x_1x_2x_3,x_2^2x_3^2,x_1x_3^3)$,
$(x_1x_2^2,x_1x_2x_3,x_1x_3^2,x_2^2x_3^3)$, 

$(x_1x_2^3,x_1x_2x_3,x_2^2x_3^2,x_1x_3^3)$, 
$(x_1x_2^3,x_1x_2x_3,x_1x_3^2,x_2^2x_3^3)$, 
$(x_1x_2^2,x_1x_2x_3,x_1x_3^2,x_2^3x_3^3)$, 

$(x_1^2x_2^2,x_1x_2^2x_3,x_1x_2x_3^2,x_3^3)$,
$(x_1^2x_2^3,x_1x_2^2x_3,x_1x_2x_3^2,x_3^3)$,
$(x_1^2x_2^2,x_1x_2x_3,x_1x_3^2,x_2x_3^2)$, 

$(x_1^2x_2^2,x_1x_2x_3,x_2x_3^2,x_1x_3^3)$, 
$(x_1^2x_2^2,x_1x_2^2x_3,x_1x_3^2,x_2x_3^2)$, 
$(x_1^2x_2^2,x_1x_2^2x_3,x_2x_3^2,x_1x_3^3)$, 

$(x_1^2x_2^2,x_1x_2^2x_3,x_1x_3^2,x_2x_3^3)$, 
$(x_1^2x_2^3,x_1x_2^2x_3,x_1x_3^2,x_2x_3^2)$, 
$(x_1^2x_2^3,x_1x_2^2x_3,x_2x_3^2,x_1x_3^3)$, 

$(x_1^2x_2^3,x_1x_2^2x_3,x_1x_3^2,x_2x_3^3)$, 
$(x_1x_2^2,x_1^2x_3,x_1x_2x_3,x_2x_3^2)$, 
$(x_1x_2^2,x_1x_2x_3,x_1^2x_3^2,x_2x_3^2)$, 

$(x_1x_2^2,x_1x_2x_3,x_2x_3^2,x_1^2x_3^3)$,
$(x_1x_2^2,x_1x_2x_3,x_1^2x_3^2,x_2x_3^3)$,
$(x_1x_2^3,x_1x_2^2x_3,x_1^2x_3^2,x_2x_3^2)$, 

$(x_1x_2^3,x_1x_2^2x_3,x_2x_3^2,x_1^2x_3^3)$, 
$(x_1^2x_2^2,x_1x_2x_3,x_1x_3^2,x_2^2x_3^2)$, 
$(x_1^2x_2^2,x_1x_2x_3,x_2^2x_3^2,x_1x_3^3)$, 

$(x_1^2x_2^2,x_1x_2x_3,x_1x_3^2,x_2^2x_3^3)$, 
$(x_1^2x_2^3,x_1x_2x_3,x_1x_3^2,x_2^2x_3^2)$, 
$(x_1^2x_2^3,x_1x_2x_3,x_2^2x_3^2,x_1x_3^3)$, 

$(x_1^2x_2^3,x_1x_2x_3,x_1x_3^2,x_2^2x_3^3)$, 
$(x_1x_2^2,x_1x_2x_3,x_1^2x_3^2,x_2^2x_3^3)$, 
$(x_1x_2^3,x_1x_2x_3,x_1^2x_3^2,x_2^2x_3^3)$, 

$(x_1^2x_2^2,x_1x_2x_3,x_1x_3^2,x_2^3x_3^3)$, 
$(x_1^2x_2^3,x_1^2x_2^2x_3,x_1x_2x_3^2,x_3^3)$, 
$(x_1^2x_2^3,x_1^2x_2^2x_3,x_1x_3^2,x_2x_3^2)$, 

$(x_1^2x_2^3,x_1^2x_2^2x_3,x_2x_3^2,x_1x_3^3)$, 
$(x_1^2x_2^3,x_1^2x_2^2x_3,x_1x_3^2,x_2x_3^3)$, 
$(x_1^2x_2^2,x_1x_2^2x_3,x_1^2x_3^2,x_2x_3^2)$, 

$(x_1^2x_2^2,x_1x_2^2x_3,x_2x_3^2,x_1^2x_3^3)$, 
$(x_1^2x_2^3,x_1x_2^2x_3,x_1^2x_3^2,x_2x_3^2)$, 
$(x_1^2x_2^3,x_1x_2^2x_3,x_2x_3^2,x_1^2x_3^3)$, 

$(x_1^2x_2^3,x_1^2x_2x_3,x_1x_3^2,x_2^2x_3^2)$, 
$(x_1^2x_2^3,x_1^2x_2x_3,x_1x_3^2,x_2^2x_3^3)$, 
$(x_1^2x_2^2,x_1x_2x_3,x_1^2x_3^2,x_2^2x_3^2)$, 

$(x_1^2x_2^2,x_1x_2x_3,x_2^2x_3^2,x_1^2x_3^3)$, 
$(x_1^2x_2^3,x_1x_2x_3,x_2^2x_3^2,x_1^2x_3^3)$, 
$(x_1^2x_2^3,x_1^2x_3,x_1x_2x_3^2,x_2^2x_3^2)$, 

$(x_1^2x_2^3,x_1x_2x_3,x_1^2x_3^2,x_2^2x_3^3)$, 
$(x_1^2x_2^3,x_1^2x_3,x_1x_2x_3^2,x_2^2x_3^3)$, 
$(x_1^2x_2^2,x_1^2x_2x_3,x_1x_3^2,x_2^3x_3^3)$, 

$(x_1^2x_2^2,x_1x_2x_3,x_1^2x_3^2,x_2^3x_3^3)$, 
$(x_1^2x_2^2,x_1^2x_3,x_1x_2x_3^2,x_2^3x_3^3)$, 
$(x_1^2x_2,x_1^2x_3,x_1x_2^2x_3^2,x_2^3x_3^3)$, 

$(x_1^3x_2^3,x_1^2x_2^2x_3,x_2x_3^2,x_1x_3^3)$, 
$(x_1^3x_2^3,x_1x_2^2x_3,x_2x_3^2,x_1^2x_3^3)$, 
$(x_1^2x_2^3,x_1x_2^2x_3,x_2x_3^2,x_1^3x_3^3)$, 

$(x_1^3x_2^3,x_1x_2x_3,x_2^2x_3^2,x_1^2x_3^3)$, 
$(x_1^2x_2^3,x_1x_2x_3,x_1^3x_3^2,x_2^2x_3^3)$.

\medskip

(d) In the list of Cohen--Macaulay ideals given in (a) to (c), the
only Gorenstein ideal is $(x_1,x_2,x_3)$ which is not surprising 
since it is known that graded Gorenstein ideals of codimension $2$ are complete
intersections \cite[p.~119, 3.3.25(d)]{BHer}. 
\end{appendix}

\section*{Statements and Declarations}  
On behalf of all authors, the corresponding author states that there is no conflict of interest.

No funding was received for conducting this study.

The authors have no relevant financial or non-financial interests to disclose.

Data sharing is not applicable to this article as no datasets were
generated or analyzed during the current study.

\bibliographystyle{plain}

\end{document}